\documentclass[11pt,reqno]{amsart}
\usepackage[margin=1.15in]{geometry}
\usepackage{amsmath,amssymb,amsthm,enumitem,mathtools,booktabs}
\usepackage{microtype}
\usepackage{hyperref}
\usepackage[numbers,sort&compress]{natbib}

\newcommand{\depth}{\operatorname{depth}}
\newcommand{\reg}{\operatorname{reg}}
\newcommand{\ind}{\operatorname{Ind}}
\newcommand{\pd}{\operatorname{pd}}
\newcommand{\link}{\operatorname{link}}

\newcommand{\cf}{\mathsf{MF}}

\newcommand{\K}{\mathbb{K}}

\DeclareMathOperator{\girth}{girth}
\DeclareMathOperator{\circum}{circ}

\newtheorem{theorem}{Theorem}[section]
\newtheorem{definition}[theorem]{Definition}

\newtheorem{lemma}[theorem]{Lemma}
\newtheorem{proposition}[theorem]{Proposition}
\newtheorem{corollary}[theorem]{Corollary}
\newtheorem{example}[theorem]{Example}

\newtheorem{remark}[theorem]{Remark}
\newtheorem{question}[theorem]{Question}

\newtheorem{conj}[theorem]{Conjecture}

\newtheorem*{notation*}{Notation}

\begin{document}

\title[Squarefree powers of very well-covered graphs]
{Cohen--Macaulay squarefree powers of edge ideals of very well-covered graphs}

\author{Rakesh Ghosh}
\email{rakeshghosh1591@gmail.com}
\address{Department Of Mathematics, Indian Institute Of Technology  Bhubaneswar, Bhubaneswar, 752050, India}

\author{S Selvaraja}
\email{selvas@iitbbs.ac.in}
\address{Department Of Mathematics, Indian Institute Of Technology  Bhubaneswar, Bhubaneswar, 752050, India}

\thanks{AMS Classification 2020: 13F55, 05E40, 13C14, 13D02, 05C70, 05E45.}

\begin{abstract}
Let $G$ be a graph with edge ideal $I(G)$. For $q\ge 1$, the $q$th squarefree power
$I(G)^{[q]}$ is the Stanley--Reisner ideal of the $q$-matching-free complex $\cf^q(G)$,
whose faces are the vertex sets containing no $q$ pairwise disjoint edges. We determine
exactly which squarefree powers of the edge ideal of a very well-covered graph are
Cohen--Macaulay: if $G$ has $2n$ vertices, then $R/I(G)^{[q]}$ is Cohen--Macaulay if and
only if $q=n$, or $q=1$ and $G$ is Cohen--Macaulay, or $G=W(H)$ is a whisker graph and
$q\le\lfloor\girth(H)/2\rfloor$. The answer is independent of the field. The proof rests
on a rigidity theorem: among very well-covered graphs, only whisker graphs have a
sequentially Cohen--Macaulay squarefree power strictly between the first and the last.
For whisker graphs we show that $\cf^q(W(H))$ is vertex decomposable for
$q\le\lceil\girth(H)/2\rceil$, and we express the depth and the regularity of every
squarefree power through the homology of explicit auxiliary complexes. As a consequence
we prove a conjecture of Das, Roy and Saha on the depth of squarefree powers of whisker
cycles. Finally, for every graph with a perfect matching, we compute the top homology of
$\cf^q(G)$ and determine when the regularity of $R/I(G)^{[q]}$ attains its maximal
possible value.
\end{abstract}

\maketitle
\section{Introduction}

Powers of ideals rarely inherit good homological properties. For Stanley--Reisner ideals this is
made precise by two classical results: by Cowsik and Nori~\cite{CN76}, a radical ideal all of whose
powers are Cohen--Macaulay is a complete intersection, and Terai and Trung~\cite{TT12} showed that
for squarefree monomial ideals a single Cohen--Macaulay power $I^q$ with $q\ge3$ already forces this.
\emph{Squarefree powers}, introduced in~\cite{BHZ18,EHHM22}, offer a far more flexible alternative.
For the edge ideal $I(G)=(x_ux_v:\{u,v\}\in E(G))$ of a graph $G=(V(G), E(G))$ in $R=\K[V(G)]$, the $q$th
squarefree power $I(G)^{[q]}$ is generated by the products $\prod_{v\in e_1\cup\cdots\cup e_q}x_v$
over all sets of $q$ pairwise disjoint edges; it vanishes for $q>\nu(G)$, where $\nu(G)$ denotes the
matching number of $G$. Being squarefree, $I(G)^{[q]}$ is itself a Stanley--Reisner ideal: it is the ideal
of the \emph{$q$-matching-free complex}
\[
\cf^q(G)=\bigl\{F\subseteq V(G)\;\big|\;\nu(G[F])\le q-1\bigr\}
\]
(Proposition~\ref{SR}). For $q=1$ this is the independence complex $\ind(G)$, a central object of
topological combinatorics~\cite{J08}; for larger $q$ it records how far a vertex set is from
containing $q$ disjoint edges. Thus squarefree powers connect commutative algebra, matching theory
and combinatorial topology, and their Cohen--Macaulayness, regularity, depth and linear quotients
have been studied intensively in recent years; see, for example,
\cite{CF26,CFE25,DRS24,DRS25, DGS26,EF26,EHHS24,EH21,EH25,FM24,HS25,Sey24,Fa25}.

A first line of work asks when all squarefree powers are Cohen--Macaulay. Das, Roy and
Saha~\cite{DRS24} showed that this holds for Cohen--Macaulay forests, and~\cite{DRS25} extended this
to facet ideals of Cohen--Macaulay simplicial forests. Ficarra and Moradi~\cite{FM24} characterized
the chordal, Cameron--Walker and very well-covered graphs with this property, and proved that the
last power $I(G)^{[\nu(G)]}$ is Cohen--Macaulay if and only if $G$ has a perfect matching or is
factor-critical. The emerging picture is that graphs all of whose squarefree powers are
Cohen--Macaulay are rare. This calls for a finer question, which we call the \emph{range problem}:
\begin{quote}
\emph{for a fixed graph $G$, for which $q$ is $R/I(G)^{[q]}$ Cohen--Macaulay?}
\end{quote}
To the best of our knowledge, no complete answer is known for any large class of graphs.

We solve the range problem for \emph{very well-covered graphs}: the well-covered graphs $G$ without
isolated vertices with $\operatorname{height}I(G)=|V(G)|/2$. This classical family contains the
unmixed bipartite graphs and the \emph{whisker graphs} $W(H)$, obtained from a graph $H$ by attaching
a pendant edge to every vertex. Its Cohen--Macaulay members were characterized by Crupi, Rinaldo and
Terai~\cite{crupi}, and its edge and cover ideals have been studied in depth~\cite{mohammad,KTY,CF23,CF24}.
A very well-covered graph on $2n$ vertices has a perfect matching, so $\nu(G)=n$ and $I(G)^{[n]}$ is
always Cohen--Macaulay; Ficarra and Moradi~\cite[Corollary~2.8]{FM24} showed that all powers
are Cohen--Macaulay only for Cohen--Macaulay forests. Our main theorem determines exactly which powers are
Cohen--Macaulay.

\begin{theorem}[Corollary~\ref{cor:vwc-complete}]\label{intro-main}
Let $G$ be a very well-covered graph on $2n$ vertices and let $1\le q\le n$. Then, over every
field, $R/I(G)^{[q]}$ is Cohen--Macaulay if and only if one of the following holds:
\begin{enumerate}
\item $G=W(H)$ for a forest $H$;
\item $G=W(H)$ for a graph $H$ with a cycle, and either $q\le\lfloor\girth(H)/2\rfloor$ or $q=n$;
\item $G$ is Cohen--Macaulay but not a whisker graph, and $q\in\{1,n\}$;
\item $G$ is not Cohen--Macaulay, and $q=n$.
\end{enumerate}
\end{theorem}

In a single sentence: $R/I(G)^{[q]}$ is Cohen--Macaulay if and only if $q=n$, or $q=1$ and $G$ is
Cohen--Macaulay, or $G=W(H)$ and $q\le\lfloor\girth(H)/2\rfloor$, with $\girth(H)=\infty$ for a
forest. Two features of this answer stand out. First, the property is independent of the field,
although Cohen--Macaulayness of squarefree powers is a homological condition. Second, between the
first and the last power the whole range is controlled by one classical invariant, the girth of
$H$: the Cohen--Macaulay powers of $I(W(H))$ form the initial segment $1\le q\le\lfloor\girth(H)/2\rfloor$,
followed by a gap that ends only at $q=n$. For a whisker cycle $W(C_n)$, for instance, exactly the
powers $q\le\lfloor n/2\rfloor$ and $q=n$ are Cohen--Macaulay.

The proof has two independent parts, each of interest in its own right.

\emph{Rigidity of very well-covered graphs.} For an arbitrary very well-covered graph we describe the
facets of $\cf^q(G)$ and show that $\cf^q(G)$ is pure precisely when $q<\lceil\ell/2\rceil$ or
$q>n-\lfloor\ell/2\rfloor$, where $\ell$ is the length of a shortest odd cycle of $G$
(Theorem~\ref{thm:vwc-pure}). Purity alone cannot single out whisker graphs. What does is a rigidity
phenomenon, which holds even for the weaker property of sequential Cohen--Macaulayness: if $n\ge3$,
then some power $I(G)^{[q]}$ with $2\le q\le n-1$ is sequentially Cohen--Macaulay if and only
if $G$ is a whisker graph, and in that case $\cf^2(G)$ is vertex decomposable
(Theorems~\ref{thm:rigid} and~\ref{thm:noncm-vwc}, Corollary~\ref{cor:all-vwc}). This sharpens the
result of Ficarra and Moradi in two directions: a single intermediate power suffices, and
Cohen--Macaulay may be weakened to sequentially Cohen--Macaulay. The obstruction is local and
explicit: we find a face of $\cf^2(G)$ whose link has edges in two different connected components
(Lemma~\ref{lem:two-edges}), with one exceptional configuration, a $C_4$-component, handled
separately (Lemma~\ref{lem:C4-comp}); links of pairs of the perfect matching then transport the
obstruction to every $2\le q\le n-1$ (Corollary~\ref{cor:pair-link}).

\emph{Whisker graphs and the girth.} For $q=1$, Villarreal~\cite{vill_cohen} proved that
$\cf^1(W(H))$ is Cohen--Macaulay, and Dochtermann and Engstr\"om~\cite{DochEng09} showed that it is
vertex decomposable. We extend this to the whole range governed by the girth: $\cf^q(W(H))$ is vertex
decomposable for every $q\le\lceil\girth(H)/2\rceil$ (Theorem~\ref{reg:main}), and the bound cannot be improved when $H$ has girth at least $4$
(Example~\ref{sharp-bd} and Corollary~\ref{cor:scm-trianglefree}). The key device is a family of \emph{cost complexes}
$\Theta(K,s)=\{T\subseteq V(K):|T|+\nu(K\setminus T)\le s\}$, which appear as links of
$\cf^q(W(H))$ and whose shedding vertices are leaves and isolated vertices of $K$; the girth
condition guarantees that such vertices exist whenever they are needed. Beyond the girth range,
Cohen--Macaulayness fails for a transparent reason: an induced cycle of length $2m$ in $H$ yields a
face of $\cf^q(W(H))$ whose link is a disjoint union of two $(m-1)$-simplices
(Proposition~\ref{prop:evencycle}), while odd cycles destroy purity and long cycles force the depth
down to its minimum (Theorem~\ref{thm:cm-range}).

Once Cohen--Macaulayness fails, the natural measure of the failure is the depth. For an arbitrary
graph $H$ we compute $\depth R/I(W(H))^{[q]}$ and $\reg R/I(W(H))^{[q]}$ in terms of the homology of
the complexes
\[
\Gamma_r(K)=\{D\subseteq V(K):|D|-\nu(K[D])\le r\},
\]
where $K$ runs over induced subgraphs of $H$ with isolated vertices added
(Theorem~\ref{thm:general-depth-reg}). These complexes are Alexander dual to the cost complexes, and
cycles of $H$ produce nonzero top homology in them through an explicit exterior-algebra construction
(Lemma~\ref{lem:cycle}). This yields closed formulas in terms of the girth and the circumference
$\circum(H)$, the length of a longest cycle ($\circum(H)=0$ for a forest).

\begin{theorem}[Theorem~\ref{depth-formula}]\label{intro-depth}
Let $G=W(H)$ with $|V(H)|=n$, and let $1\le q\le n$. Then
\begin{enumerate}
\item $2q-1\le\depth R/I(G)^{[q]}\le n+q-1$;
\item if $q=n$ or $2(n-q)<\circum(H)$, then $\depth R/I(G)^{[q]}=2q-1$;
\item if $q\le\lfloor\girth(H)/2\rfloor$, then $\depth R/I(G)^{[q]}=n+q-1$; in particular this holds for every $q\le n$ if $H$ is a forest;
\item if $\girth(H)$ is odd and $q=(\girth(H)+1)/2$, then $\depth R/I(G)^{[q]}=n$.
\end{enumerate}
\end{theorem}

Thus the depth is maximal throughout the girth range and minimal, equal to the lower bound
$2q-1$ forced by the degrees of the generators, as soon as $H$ has a cycle longer than $2(n-q)$.
If $H=C_n$ is a cycle, then $\girth(H)=\circum(H)=n$, so parts~(2) and~(3) together cover every $q$, and we obtain a
proof of the following conjecture (Corollary~\ref{cor:cycle-depth}).

\begin{conj}[{\cite[Conjecture~6.3]{DRS24}}]\label{conj-DRS}
Let $G=W(C_n)$ with $n\ge3$. Then
\[
\depth\bigl(R/I(G)^{[q]}\bigr)=
\begin{cases}
n+q-1, & \text{if } 1\le q\le\lfloor n/2\rfloor,\\
2q-1, & \text{if } \lfloor n/2\rfloor+1\le q\le n.
\end{cases}
\]
\end{conj}

In particular, the depth function $q\mapsto\depth R/I(W(C_n))^{[q]}$ is not monotone, whereas the
normalized depth function of Erey, Herzog, Hibi and Saeedi~Madani is non-increasing
(Remark~\ref{rem:monotone}). For Cohen--Macaulay very well-covered graphs that are not whisker graphs,
the depth drops to the minimal value $2q-1$ for all large $q$ (Theorem~\ref{thm:nonwhisker-depth} and
Corollary~\ref{cor:depth-n-1}).

Our last main result holds for every graph with a perfect matching, and it is purely
topological: it computes the homology of $\cf^q(G)$ in the highest possible degree.

\begin{theorem}[Theorem~\ref{thm:pm-top} and Corollary~\ref{cor:pm-reg}]\label{intro-pm}
Let $G$ be a graph on $2n$ vertices with a perfect matching and with $k$ connected components, and let $1\le q\le n$. Then
\[
\dim_\K\widetilde H_{n+q-2}\bigl(\cf^q(G);\K\bigr)=\beta_{n-q+1,2n}\bigl(R/I(G)^{[q]}\bigr)=\binom{k-1}{n-q}.
\]
Moreover, $\reg R/I(G)^{[q]}=n+q-1$ if and only if $q\ge n-k+1$; in this case $\beta_{n-q+1,2n}(R/I(G)^{[q]})$ is an extremal Betti number.
\end{theorem}

Seyed Fakhari~\cite{Sey24,Fa25} proved the upper bound $\reg R/I(G)^{[q]}\le n+q-1$ for these
graphs; Theorem~\ref{intro-pm} shows that it is attained exactly when $q\ge n-k+1$, so the answer
depends only on the number of connected components. For very well-covered graphs this decides when the
regularity equals the dimension. In the same range a question of Chau, Das, Roy and
Saha~\cite[Question~6.1]{CDRS25} has an affirmative answer, which there already follows from known bounds (Remark~\ref{rem:aim-question}). Combined with
shellability, Theorem~\ref{intro-pm} also determines the homotopy type of $\cf^q(W(H))$ throughout the
Cohen--Macaulay range: it is a wedge of $\binom{k-1}{n-q}$ spheres of dimension $n+q-2$, and it is
collapsible if and only if $q\le n-k$ (Theorem~\ref{thm:homotopy}).

Francisco and Van Tuyl~\cite[Proposition~4.1]{FVT07} showed that $R/I(C_n)$ (or $\cf^1(C_n)$) is sequentially Cohen--Macaulay if and only if $n\in\{3,5\}$. For squarefree powers of whisker graphs, every \emph{hole} of $H$ (an induced cycle of
length at least $4$) obstructs sequential Cohen--Macaulayness, whereas triangles do not: a hole of
length $\ell$ rules out all $\lceil\ell/2\rceil<q<n$ (Theorem~\ref{thm:holes}). For triangle-free $H$
this determines the sequentially Cohen--Macaulay range completely
(Corollary~\ref{cor:scm-trianglefree}), and it shows that $H$ must be chordal if all squarefree powers
of $I(W(H))$ are sequentially Cohen--Macaulay (Corollary~\ref{cor:chordal}). 
We ask whether
chordality is also sufficient (Question~\ref{conj:hole}). Finally, we pose a question on the
homology of the complexes $\Gamma_r$ (Question~\ref{conj:conc}), which we expect to imply that the
normalized depth function of $I(W(H))$ is non-increasing for every graph $H$.

\vskip 2mm
\noindent
\emph{Organization:}
Section~\ref{pre} collects background. Section~\ref{pure} describes the facets of $\cf^q(G)$ for very
well-covered graphs and characterizes purity. Section~\ref{shellable} proves vertex decomposability
for whisker graphs in the girth range, and Section~\ref{sec:rigid} establishes the rigidity of very
well-covered graphs. Section~\ref{depth-sq} develops the depth formula, determines the
Cohen--Macaulay range and completes the proof of Theorem~\ref{intro-main}.
Section~\ref{sec:scm} treats sequential Cohen--Macaulayness, Section~\ref{sec:topology} the top
homology and the regularity.

\section{Preliminaries}\label{pre}

In this section we fix notation and recall the definitions and standard facts used throughout the paper.
All graphs are finite and simple. For a graph $G$, we denote by $V(G)$ its vertex set and by $E(G)$ its edge set. For $S\subseteq V(G)$ we write $G[S]$ for the induced subgraph on $S$, $G\setminus S=G[V(G)\setminus S]$, and $G\setminus v=G\setminus\{v\}$. The \emph{neighborhood} of $S$ is $N_G(S)=\{v\in V(G)\setminus S:\{u,v\}\in E(G)\text{ for some }u\in S\}$ and $N_G[S]=N_G(S) \cup S$.
The \emph{girth} of a graph $G$, denoted by
$\girth(G)$, is the length of a shortest cycle
in $G$, while the \emph{circumference}, denoted
by $\circum(G)$, is the length of a longest
cycle in $G$. By convention, if $G$ is a
forest, we set $\girth(G)=\infty$ and
$\circum(G)=0$.
A \emph{matching} is a set of pairwise disjoint edges, and the \emph{matching number} $\nu(G)$ is the largest size of a matching. For a matching $M$ we write $V(M)=\bigcup_{e\in M}e$.

A \emph{simplicial complex} $\Delta$ on a finite set $V$ is a family of subsets of $V$ (\emph{faces}) closed under taking subsets. We allow the \emph{void} complex (no faces) and the complex $\{\emptyset\}$. The maximal faces of $\Delta$ are its \emph{facets}. We put $\dim F=|F|-1$ and $\dim\Delta=\max\{\dim F:F\in\Delta\}$. The complex $\Delta$ is \emph{pure} if all facets have the same dimension, and a \emph{simplex} if it has exactly one facet (so $\{\emptyset\}$ is a simplex, while the void complex is not). The \emph{pure $i$-skeleton} $\Delta^{[i]}$ is the subcomplex generated by the faces of dimension exactly $i$. For a simplicial complex $\Delta$ on $V$ and a subset $\sigma\subseteq V$, the \emph{induced subcomplex} on $\sigma$ is $\Delta_\sigma=\{F\in\Delta:F\subseteq\sigma\}$.

For a vertex $v$ and a face $F\in\Delta$ we put
\[
\Delta\setminus v=\{\rho\in\Delta:v\notin\rho\},~
\link_\Delta(F)=\{F'\subseteq V:F'\cap F=\emptyset,\ F'\cup F\in\Delta\}.
\]
The \emph{join} of complexes $\Delta_1,\Delta_2$ on disjoint vertex sets is $\Delta_1*\Delta_2=\{\sigma_1\cup\sigma_2:\sigma_i\in\Delta_i\}$; for a new vertex $w$, $w*\Delta$ is the \emph{cone} over $\Delta$. 

Let $\K$ be a field. We write $\widetilde H_i(\Delta;\K)$ for reduced simplicial homology, with $\widetilde H_{-1}(\{\emptyset\};\K)=\K$.

Let $\Delta$ be a simplicial complex on a finite set $V$ (not every element of $V$ need
be a vertex of $\Delta$). The \emph{Alexander dual} of $\Delta$ with respect to $V$ is
$\Delta^\vee=\{V\setminus F : F\subseteq V,\ F\notin\Delta\}$. In particular,
$(2^V)^\vee$ is the void complex and the dual of the void complex is $2^V$; in general,
$(\Delta^\vee)^\vee=\Delta$ for every $\Delta$.

\begin{theorem}[Alexander duality; see {\cite[Theorem~5.6]{ms2005}}]\label{thm:alex}
Let $\Delta$ be a simplicial complex on $V$ with $|V|=m$ and $\Delta\ne 2^V$. Then
$\widetilde H_i(\Delta^\vee;\K)\cong\widetilde H^{\,m-i-3}(\Delta;\K)$ for all $i$. In particular,
since reduced homology and cohomology over a field have the same dimension,
$\dim_\K\widetilde H_i(\Delta^\vee;\K)=\dim_\K\widetilde H_{m-i-3}(\Delta;\K)$ for all $i$.
\end{theorem}

\begin{definition}\label{def:collapse}
Let $\Delta$ be a simplicial complex. A face $\sigma\in\Delta$ is \emph{free} if it is
not a facet and is contained in exactly one facet $\tau$ of $\Delta$. For such a pair,
the interval $[\sigma,\tau]=\{\rho:\sigma\subseteq\rho\subseteq\tau\}$ is a Boolean
lattice, and
\[
\Delta'=\Delta\setminus[\sigma,\tau]=\{\rho\in\Delta:\ \sigma\not\subseteq\rho\}
\]
is again a simplicial complex. We say that $\Delta'$ is obtained from $\Delta$ by an
\emph{elementary collapse} and write $\Delta\searrow^{e}\Delta'$. In the most frequently
used case $|\tau\setminus\sigma|=1$, an elementary collapse removes exactly the two faces
$\sigma$ and $\tau=\sigma\cup\{v\}$. A finite sequence of elementary collapses from
$\Delta$ to $\Delta''$ is called a \emph{collapse} and is denoted $\Delta\searrow\Delta''$.
The complex $\Delta$ is \emph{collapsible} if $\Delta\searrow 2^{\{v\}}=\{\emptyset,\{v\}\}$
for some vertex $v$, that is, if $\Delta$ can be reduced to a single point by finitely
many elementary collapses.
\end{definition}

\begin{remark}\label{rem:collapse}
Every simplex is collapsible. An elementary collapse does not change the homotopy
type, so a collapsible complex is contractible, and in particular
$\widetilde H_*(\Delta;\K)=0$ over every field $\K$. The converse fails in general:
there exist contractible complexes that are not collapsible.
\end{remark}

\begin{definition}[\cite{ProvLouis,BW97}]\label{def:vd}
A simplicial complex $\Delta$ is \emph{vertex decomposable} if it is a simplex, or there is a vertex $v$ \textup{(}a \emph{shedding vertex}\textup{)} such that
\begin{enumerate}
\item $\Delta\setminus v$ and $\link_\Delta(v)$ are vertex decomposable, and
\item no face of $\link_\Delta(v)$ is a facet of $\Delta\setminus v$.
\end{enumerate}
\end{definition}

Since $\link_\Delta(v)\subseteq\Delta\setminus v$, condition~(2) of Definition~\ref{def:vd} holds
for $v$ as soon as no facet of $\link_\Delta(v)$ is a facet of $\Delta\setminus v$: a face of
$\link_\Delta(v)$ that is a facet of $\Delta\setminus v$ is contained in a facet of
$\link_\Delta(v)$, which then equals it. This is how shedding vertices are verified below. A simplex, in particular $\{\emptyset\}$, is vertex decomposable, whereas the void complex is not; whenever Definition~\ref{def:vd} is applied below, the deletion and the link are checked to be nonvoid.

\begin{definition}[\cite{BW96,BW97}]
A simplicial complex $\Delta$ is \emph{shellable} if its facets admit an ordering $F_1,\ldots,F_t$ such that, for every $k\ge 2$, the intersection
$\left(\bigcup_{i<k}2^{F_i}\right)\cap 2^{F_k}$
is pure of dimension $\dim F_k-1$, where $2^F$ denotes the simplex on $F$.
\end{definition}

We use the following standard facts.

\begin{theorem}\label{thm:standard}
The following hold.
\begin{enumerate}
\item A cone over a vertex decomposable complex is vertex decomposable \cite[Proposition~2.4]{ProvLouis}, and every vertex decomposable complex is shellable \cite[Theorem~11.3]{BW97}.
\item \cite[Theorem~4.1]{BW96} A shellable complex is homotopy equivalent to a wedge of spheres; for each $i$, the number of $i$-spheres is the number of $i$-dimensional facets $F_j$ of a shelling with $\operatorname{Rest}(F_j)=F_j$, where the \emph{restriction face} $\operatorname{Rest}(F_j)$ is the set of vertices $v\in F_j$ such that $F_j\setminus\{v\}$ is contained in an earlier facet.
\end{enumerate}
\end{theorem}

Let $R=\K[x_v:v\in V]$ be the polynomial ring on the vertex set $V$ of $\Delta$. The \emph{Stanley--Reisner ideal} is $I_\Delta=(\prod_{v\in F}x_v:F\subseteq V,\ F\notin\Delta)$, and $\K[\Delta]=R/I_\Delta$; then $\dim\K[\Delta]=\dim\Delta+1$.
The complex $\Delta$ is \emph{Cohen--Macaulay over $\K$} if $\K[\Delta]$ is Cohen--Macaulay; by
Reisner's criterion \cite[Corollary~5.3.9]{BH93} this holds if and only if $\widetilde H_i(\link_\Delta(\sigma);\K)=0$ for all
$\sigma\in\Delta$ and all $i<\dim\link_\Delta(\sigma)$, and then $\Delta$ is pure. It is
\emph{sequentially Cohen--Macaulay over $\K$} if $\K[\Delta]$ is sequentially Cohen--Macaulay in the
sense of Stanley \cite[Section~III.2]{Stanbook}; by a theorem of Duval \cite[Theorem~3.3]{Duval}
this holds if and only if every pure skeleton $\Delta^{[i]}$ is Cohen--Macaulay over $\K$. Shellable
complexes are sequentially Cohen--Macaulay over every field \cite[Section~III.2]{Stanbook}, and pure
shellable complexes are Cohen--Macaulay. A squarefree monomial ideal $I\subseteq R$ is called
\emph{\textup{(}sequentially\textup{)} Cohen--Macaulay} if $R/I$ is; by
Proposition~\ref{SR}, $I(G)^{[q]}$ is \textup{(}sequentially\textup{)} Cohen--Macaulay if and only
if $\cf^q(G)$ is.

\begin{lemma}[{\cite{BWW09}}]\label{lem:scm-links}
If $\Delta$ is sequentially Cohen--Macaulay over $\K$, then so is $\link_\Delta(F)$ for every face
$F\in\Delta$.
\end{lemma}
Indeed, in \cite[Definition~1.2]{BWW09} $\Delta$ is called sequentially Cohen--Macaulay if every
link of $\Delta$ is sequentially acyclic, and this agrees with the definition above
\cite[Proposition~2.5]{BWW09}; since a link of a link of $\Delta$ is a link of $\Delta$, the lemma
follows.

\begin{lemma}\label{lem:scm-conn}
Let $\Lambda$ be a sequentially Cohen--Macaulay simplicial complex. Then all $1$-dimensional
faces of $\Lambda$ lie in the same connected component of $\Lambda$.
\end{lemma}
\begin{proof}
The pure skeleton $\Lambda^{[1]}$ is Cohen--Macaulay of dimension one, hence connected.
\end{proof}

\begin{lemma}[{\cite{Faridi13}}]\label{lem:seqCM-depth}
If $\Delta$ is sequentially Cohen--Macaulay, then $\depth\K[\Delta]$ equals the smallest cardinality of a facet of $\Delta$.
\end{lemma}

The graded Betti numbers of a graded module $M$ over $R$ are $\beta_{i,j}(M)=\dim_\K\operatorname{Tor}_i^R(M,\K)_j$; $\pd(M)=\max\{i:\beta_{i,j}(M)\ne0\}$ and $\reg(M)=\max\{j-i:\beta_{i,j}(M)\ne0\}$. A nonzero $\beta_{i,i+j}(M)$ is \emph{extremal} if $\beta_{k,k+\ell}(M)=0$ for all $(k,\ell)\ne(i,j)$ with $k\ge i$ and $\ell\ge j$ \cite{BCP99}.

\begin{theorem}[Hochster \cite{H77}; Auslander--Buchsbaum]\label{HT-result}
Let $\Delta$ be a simplicial complex on $V$ with $|V|=s$, and $R=\K[x_v:v\in V]$.
\begin{enumerate}
\item For $\sigma\subseteq V$ and $i\ge0$, $\beta_{i,\sigma}(R/I_\Delta)=\dim_\K\widetilde H_{|\sigma|-i-1}(\Delta_\sigma;\K)$. In particular $\reg(R/I_\Delta)=\max\{j+1:\widetilde H_j(\Delta_\sigma;\K)\ne0\text{ for some }\sigma\}$.
\item For a finitely generated graded $R$-module $M$, $\pd_R(M)+\depth(M)=s$.
\end{enumerate}
\end{theorem}

\begin{definition}
For a graph $G$ and $q\ge1$, the \emph{$q$-matching-free complex} of $G$ is
\[
\cf^q(G)=\{F\subseteq V(G):\nu(G[F])\le q-1\}.
\]
\end{definition}

If $F'\subseteq F$ then $\nu(G[F'])\le\nu(G[F])$, so $\cf^q(G)$ is a simplicial complex on $V(G)$ containing all singletons. For $q=1$ it is the independence complex of $G$, and for $q>\nu(G)$ it is the full simplex.

The \emph{$q$th squarefree power} $I(G)^{[q]}$ of the edge ideal $I(G)$ is generated by the monomials $\prod_{i=1}^q x_{e_i}$ over all $q$-matchings $\{e_1,\dots,e_q\}$ of $G$, where $x_e=x_ux_v$ for $e=\{u,v\}$.

\begin{proposition}\label{SR}
For every graph $G$ and $q\ge1$, $I_{\cf^q(G)}=I(G)^{[q]}$.
\end{proposition}
This is well known (see, e.g., \cite[Section~2]{EHHS24}) and follows directly from the definitions:
 $F\subseteq V(G)$ is a nonface of
$\cf^q(G)$ if and only if $G[F]$ contains a $q$-matching, that is, if and only if
$\prod_{v\in F}x_v\in I(G)^{[q]}$.

\begin{remark}\label{rem:top-power}
If $G$ has $2n$ vertices and a perfect matching, then the only nonface of $\cf^n(G)$ is $V(G)$, so
$\cf^n(G)$ is the boundary of the simplex on $V(G)$ and $I(G)^{[n]}=(\prod_{v\in V(G)}x_v)$ is
principal. In particular, $R/I(G)^{[n]}$ is Cohen--Macaulay of depth $2n-1$
\textup{(}see also \cite[Remark~2.8]{DRS24} and \cite[Theorem~1.8]{FM24}\textup{)}.
\end{remark}

A graph $G$ without isolated vertices is \emph{very well-covered} if it is well-covered (that is, all its maximal independent sets have the same cardinality) and $|V(G)|=2\operatorname{height}I(G)$. We use the following description of these graphs.

\begin{theorem}[{\cite[Proposition~2.3 and Theorem~3.6]{crupi}}]\label{thm:vwc}
Let $G$ be a graph on $2n$ vertices without isolated vertices. Then $G$ is very well-covered if and only if its vertices can be labeled $x_1,\dots,x_n,y_1,\dots,y_n$ so that
\begin{enumerate}
\item[\textup{(V1)}] $\{y_1,\dots,y_n\}$ is an independent set of $G$ and $\{x_i,y_i\}\in E(G)$ for all $i$;
\item[\textup{(V2)}] if $z\in V(G)$, $\{z,x_j\}\in E(G)$ and $\{y_j,x_k\}\in E(G)$, then $\{z,x_k\}\in E(G)$;
\item[\textup{(V3)}] if $\{x_i,y_j\}\in E(G)$ and $i\ne j$, then $\{x_i,x_j\}\notin E(G)$.
\end{enumerate}
Moreover, $G$ is Cohen--Macaulay if and only if the labeling can be chosen so that, in addition,
\begin{enumerate}
\item[\textup{(V4)}] $\{x_i,y_j\}\in E(G)$ implies $i\le j$.
\end{enumerate}
\end{theorem}

In \cite[Proposition~2.3]{crupi} condition \textup{(V2)} is required only for $z\in\{x_i,y_i\}$ with $i,j,k$
pairwise distinct; this is equivalent to \textup{(V2)}, because for $z\in\{x_j,y_j,y_k\}$ the
hypothesis fails or the conclusion is $\{y_j,x_k\}\in E(G)$ or $\{y_k,x_k\}\in E(G)$, and $z=x_k$ is
excluded by \textup{(V3)}. A very well-covered graph is always considered with such a labeling; we call $\{x_i,y_i\}$ the $i$th \emph{pair}, and for a vertex $v$ we write $\overline v$ for the other vertex of its pair. By \textup{(V1)}, the pairs form a perfect matching, so $\nu(G)=n$. Conditions \textup{(V1)}--\textup{(V4)} are inherited by the graph $G\setminus\{x_i,y_i\}$, after
renumbering the remaining pairs by $1,\dots,n-1$ in their original order; moreover
$G\setminus\{x_i,y_i\}$ has no isolated vertices, since each remaining pair is an edge. Hence
$G\setminus\{x_i,y_i\}$ is again very well-covered \textup{(}and Cohen--Macaulay if $G$ is\textup{)}.

\vskip 2mm
\noindent
\textbf{Conventions.}
Throughout the paper, a graph or a simplicial complex on $[n]$ means one with vertex set $\{x_1,\dots,x_n\}$, where $i\in[n]$ is identified with the vertex $x_i$; accordingly, a subset $D\subseteq[n]$ is identified with the vertex set $\{x_i:i\in D\}$, and for a graph $H$ on $[n]$ we write $H[D]$ and $H\setminus D$ for the induced subgraph on $\{x_i:i\in D\}$ and its complement. For a forest $H$, where $\girth(H)=\infty$, we put
$\lfloor \infty/2\rfloor=\lceil\infty/2\rceil=\infty$; thus
$\min\{n,\lfloor \operatorname{girth}(H)/2\rfloor\}=n$ in that case.

For undefined terminology and further basic properties we refer to the standard
references~\cite{BH93,Herzog'sBook,J08,west}.

\section{Purity of matching-free complexes}\label{pure}
In this section, we investigate the purity of matching-free complexes associated with very
well-covered graphs, and in particular with whisker graphs. We first recall the definition of a
whisker graph.

\begin{definition}\label{def:whisker}
Let $H$ be a graph on $[n]$. The \emph{whisker graph} $W(H)$ is obtained by adjoining, for each
$i\in[n]$, a new vertex $y_i$ and the edge $\{x_i,y_i\}$, where $x_i$ denotes the vertex of $H$
corresponding to $i$. Thus
$V(W(H))=\{x_1,\ldots,x_n,y_1,\ldots,y_n\}$
and
$E(W(H))=\bigl\{\{x_i,x_j\}:\{i,j\}\in E(H)\bigr\}\cup\bigl\{\{x_i,y_i\}:i\in[n]\bigr\}.$
In particular, $H$ is identified with the induced subgraph of $W(H)$ on $\{x_1,\ldots,x_n\}$.
\end{definition}

With this labeling, $W(H)$ satisfies \textup{(V1)}--\textup{(V4)} of Theorem~\ref{thm:vwc} \textup{(}the only neighbor of $y_j$ is $x_j$\textup{)}; so whisker graphs are Cohen--Macaulay very well-covered graphs.

\begin{remark}\label{rem:whisker-subsets}
For $U,T\subseteq[n]$, write $x_U=\{x_i:i\in U\}$ and $y_T=\{y_j:j\in T\}$. Every subset
$F\subseteq V(W(H))$ can be uniquely expressed as $F=x_U\cup y_T$ for some $U,T\subseteq[n]$.
\end{remark}

The following fact is well known; we include its short proof.

\begin{lemma}\label{lem:pend}
Let $\Gamma$ be a graph, $B\subseteq V(\Gamma)$, and let $\Gamma_B$
be obtained by attaching a new pendant edge at each vertex of $B$.
Then
$\nu(\Gamma_B)=|B|+\nu(\Gamma\setminus B).$
\end{lemma}

\begin{proof}
The inequality $\nu(\Gamma_B)\ge |B|+\nu(\Gamma\setminus B)$ is
clear: take all $|B|$ pendant edges together with a maximum matching
of $\Gamma\setminus B$.

Conversely, let $M$ be a maximum matching of $\Gamma_B$. Every
$b\in B$ is covered by $M$, since otherwise adding the pendant edge
at $b$ would enlarge $M$. If $b$ is matched to a vertex other than
its pendant neighbor $p_b$, then $p_b$ is unmatched, and replacing
the edge incident to $b$ by $bp_b$ preserves the size of the matching.
Since the pendant vertices are distinct and each has degree one, these
replacements can be performed successively for all $b\in B$ without
interfering with one another. We thus obtain a maximum matching
containing all $|B|$ pendant edges. The remaining edges form a matching
of $\Gamma\setminus B$, and hence
$\nu(\Gamma_B)\le |B|+\nu(\Gamma\setminus B).$
Therefore,
$\nu(\Gamma_B)=|B|+\nu(\Gamma\setminus B).$
\end{proof}

We also need the following notions from matching theory. A graph $K$ is
\emph{factor-critical} if $K\setminus w$ has a perfect matching for every $w\in V(K)$; for example, a
single vertex and every odd cycle are factor-critical. Every factor-critical graph has odd order.

\begin{definition}\label{def:D}
Let $\Gamma$ be a graph. A vertex $v\in V(\Gamma)$ is \emph{inessential} if some maximum matching
of $\Gamma$ misses $v$ (i.e., does not cover $v$). We denote by $D(\Gamma)$ the set of inessential
vertices of $\Gamma$:
\[
  D(\Gamma)=\{\,v\in V(\Gamma) : \text{there is a maximum matching of }\Gamma\text{ not covering }v\,\}.
\]
Equivalently, $v\in D(\Gamma)$ if and only if $\nu(\Gamma\setminus v)=\nu(\Gamma)$.
\end{definition}

The following is a consequence of Gallai's lemma \cite[Section~3.1]{LP09}, applied to each connected
component.

\begin{lemma}[Gallai]
\label{lem:GE-allmissed}
For a graph $\Gamma$ the following are equivalent:
\begin{enumerate}
\item $D(\Gamma)=V(\Gamma)$, i.e., every vertex of $\Gamma$ is missed by some maximum matching;
\item every connected component of $\Gamma$ is factor-critical.
\end{enumerate}
In this case
$\nu(\Gamma)=\sum_K\frac{|V(K)|-1}{2},$
where the sum is over the connected components $K$ of $\Gamma$.
\end{lemma}

Let $G$ be a very well-covered graph on $2n$ vertices, labeled as in Theorem~\ref{thm:vwc}. As in
Remark~\ref{rem:whisker-subsets}, every $F\subseteq V(G)$ is uniquely of the form $F=x_U\cup y_T$
with $U,T\subseteq[n]$. We put
\[
A=U\cap T,~ Z=U\setminus T,~ W=T\setminus U,~ N=[n]\setminus(U\cup T),
\]
and call $\mathcal Q_F=G[x_Z\cup y_W]$ the \emph{residual graph} of $F$; it contains at most one
vertex of each pair. The following lemma extends Lemma~\ref{lem:pend} to all very well-covered
graphs and is the key tool of this section.

\begin{lemma}\label{lem:pair}
Let $G$ be a very well-covered graph and $F\subseteq V(G)$. Then
$\nu(G[F])=|A|+\nu(\mathcal Q_F).$
\end{lemma}
\begin{proof}
It suffices to show that if $x_i,y_i\in F$, then some maximum matching of $G[F]$ contains the
edge $x_iy_i$; the formula then follows by induction on $|A|$, applied to the graph
$G\setminus\{x_i,y_i\}$. Let $M$ be a maximum matching of $G[F]$ with $x_iy_i\notin M$. By
maximality, $M$ covers $x_i$ or $y_i$. If it covers exactly one of them, we replace the edge of
$M$ at that vertex by $x_iy_i$. Otherwise $y_ix_k\in M$ and $x_iz\in M$ for some $k\ne i$ and
some $z\ne y_i$, because all neighbors of $y_i$ are of the form $x_k$ by \textup{(V1)}. Since
$M$ is a matching, $z\ne x_k$, and by \textup{(V2)},
$\{z,x_k\}\in E(G)$. Replacing $y_ix_k$ and $x_iz$ by $x_iy_i$ and $x_kz$ gives a maximum
matching containing $x_iy_i$.
\end{proof}

\begin{corollary}\label{cor:pair-link}
Let $G$ be a very well-covered graph and $q\ge2$. Then
$\link_{\cf^q(G)}(\{x_i,y_i\})=\cf^{q-1}(G\setminus\{x_i,y_i\})$ for every $i$.
\end{corollary}
\begin{proof}
For $F\subseteq V(G)\setminus\{x_i,y_i\}$, Lemma~\ref{lem:pair} gives
$\nu(G[F\cup\{x_i,y_i\}])=1+\nu(G[F])$.
\end{proof}

\begin{lemma}\label{lem:vwc-facets}
Let $G$ be a very well-covered graph on $2n$ vertices and $1\le q\le n$.
\begin{enumerate}
\item $\dim\cf^q(G)=n+q-2$ \textup{(}cf.\ \cite[Lemma~2.4 and Proposition~2.7]{FM24}\textup{)}.
\item Let $F$ be a facet of $\cf^q(G)$ and $\mathcal Q=\mathcal Q_F$. Then
\begin{enumerate}
\item[(i)] $|A|+\nu(\mathcal Q)=q-1$;
\item[(ii)] $D(\mathcal Q)=V(\mathcal Q)$, i.e., every connected component of $\mathcal Q$ is
factor-critical;
\item[(iii)] if $\nu(\mathcal Q)=0$, then $N=\emptyset$.
\end{enumerate}
Moreover, $|F|=n+q-1-\nu(\mathcal Q)-|N|$.
\end{enumerate}
\end{lemma}
\begin{proof}
(1) For every face $F$ we have $|A|\le\nu(G[F])\le q-1$ by Lemma~\ref{lem:pair}, and
$|F|=2|A|+|Z|+|W|=n+|A|-|N|$; hence $|F|\le n+q-1$. Conversely, if $A\subseteq[n]$ with
$|A|=q-1$, then the residual graph of $F=x_A\cup y_{[n]}$ has no edges by \textup{(V1)}, so
$\nu(G[F])=q-1$ and $F$ is a face with $n+q-1$ elements.

(2) Let $F$ be a facet. If $\nu(G[F])\le q-2$, then $F\cup\{v\}$ is a face for every
$v\notin F$, because adding a vertex raises the matching number by at most one; so $F=V(G)$, which
is impossible since $\nu(G)=n\ge q$. Hence $\nu(G[F])=q-1$, and (i) follows from
Lemma~\ref{lem:pair}. Together with $|F|=n+|A|-|N|$ this gives the formula for $|F|$.

(ii) Let $v\in V(\mathcal Q)$ belong to the $i$th pair. Then $\overline v\notin F$, and the set
$F\cup\{\overline v\}$ has $A\cup\{i\}$ in place of $A$ and residual graph $\mathcal Q\setminus v$. As it
is not a face, $|A|+1+\nu(\mathcal Q\setminus v)\ge q$, so
$\nu(\mathcal Q\setminus v)=\nu(\mathcal Q)$ by (i). Thus $v\in D(\mathcal Q)$, and the
components of $\mathcal Q$ are factor-critical by Lemma~\ref{lem:GE-allmissed}.

(iii) Let $\nu(\mathcal Q)=0$, so that $V(\mathcal Q)$ is an independent set, and suppose that
$i\in N$. Neither $F\cup\{y_i\}$ nor $F\cup\{x_i\}$ is a face, and for both sets $A$ is
unchanged. Hence $y_i$ has a neighbor $x_k\in V(\mathcal Q)$ and $x_i$ has a neighbor
$z\in V(\mathcal Q)$. By \textup{(V3)}, $z\ne x_k$, and by \textup{(V2)}, $\{z,x_k\}\in E(G)$,
an edge of $\mathcal Q$; a contradiction.
\end{proof}

\begin{remark}
In contrast with whisker graphs \textup{(}see Lemma~\ref{lem:facets}\textup{)}, a facet of
$\cf^q(G)$ may avoid a pair completely: if $G$ consists of four pairs together with the triangle
$x_1x_3x_4$ and the edges $x_1y_2$, $x_2x_3$, then \textup{(V1)}--\textup{(V4)} hold and $\{x_1,x_3,x_4\}$ is a facet
of $\cf^2(G)$.
\end{remark}

\begin{lemma}\label{lem:vwc-oddcycle}
Let $G$ be a very well-covered graph. An induced odd cycle of $G$ contains at most one vertex
of each pair.
\end{lemma}
\begin{proof}
Let $C$ be an induced odd cycle with $x_i,y_i\in V(C)$. Since $C$ is induced, $x_iy_i$ is an
edge of $C$. Let $a$ be the other neighbor of $x_i$ on $C$ and let $x_k$, $k\ne i$, be the
other neighbor of $y_i$ on $C$. If $|V(C)|=3$, then $a=x_k$ is adjacent to both $x_i$ and
$y_i$, contradicting \textup{(V3)}. If $|V(C)|\ge5$, then $a\ne x_k$ and $a$, $x_k$ are not
adjacent, because $C$ is induced; but \textup{(V2)} gives $\{a,x_k\}\in E(G)$.
\end{proof}

\begin{proposition}\label{prop:vwc-smallfacet}
Let $G$ be a very well-covered graph on $2n$ vertices, and let $C$ be an induced odd cycle of
$G$ of length $2\mu+1$. If $\mu+1\le q\le n-\mu$, then $\cf^q(G)$ has a facet with at most
$n+q-1-\mu$ elements. Consequently,
\[
\depth R/I(G)^{[q]}\le n+q-1-\mu .
\]
\end{proposition}
\begin{proof}
By Lemma~\ref{lem:vwc-oddcycle}, $C$ meets exactly $2\mu+1$ pairs. Choose a set $A_0$ of
$q-1-\mu$ of the remaining $n-2\mu-1$ indices and put $F_0=V(C)\cup x_{A_0}\cup y_{A_0}$. The
residual graph of $F_0$ is $C$, so $\nu(G[F_0])=|A_0|+\mu=q-1$ by Lemma~\ref{lem:pair}, and
$F_0$ is a face. Let $F\supseteq F_0$ be a facet. If $\overline v\in F$ for some $v\in V(C)$,
then $F$ contains $F_0\cup\{\overline v\}$, and Lemma~\ref{lem:pair} gives
$\nu(G[F_0\cup\{\overline v\}])=|A_0|+1+\nu(C\setminus v)=q$, which is impossible. Hence
$V(C)\subseteq V(\mathcal Q_F)$, so $\nu(\mathcal Q_F)\ge\nu(C)=\mu$, and
$|F|\le n+q-1-\mu$ by Lemma~\ref{lem:vwc-facets}(2).
The ideal $P_F=(z:z\notin F)$ is a minimal prime of $I(G)^{[q]}$, so
$\depth R/I(G)^{[q]}\le\dim R/P_F=|F|$ by \cite[Proposition~1.2.13]{BH93}.
\end{proof}

\begin{theorem}\label{thm:vwc-pure}
Let $G$ be a very well-covered graph on $2n$ vertices, and let $1\le q\le n$. If $G$ has no odd
cycle, then $\cf^q(G)$ is pure. Otherwise, let $\ell=2k+1$ be the length of a shortest odd cycle
of $G$. Then $\cf^q(G)$ is pure if and only if
$ 1\le q<\Bigl\lceil\frac{\ell}{2}\Bigr\rceil$ or  $n-\Bigl\lfloor\frac{\ell}{2}\Bigr\rfloor<q\le n .$
\end{theorem}
\begin{proof}
If $G$ has an odd cycle, the statement says that $\cf^q(G)$ is pure if and only if $q\le k$ or
$q\ge n-k+1$.
Suppose that $\cf^q(G)$ is not pure. By Lemma~\ref{lem:vwc-facets}(1) some facet has $n+q-1$
elements, so some facet $F$ has fewer, and by the formula for $|F|$ it satisfies
$\nu(\mathcal Q)+|N|\ge1$, where $\mathcal Q=\mathcal Q_F$, and then $\nu(\mathcal Q)\ge1$ by
part~(iii) of that lemma. By part~(ii), all components of $\mathcal Q$ are factor-critical,
hence of odd order, and some component $K$ has $|K|\ge3$. Such a $K$ is not bipartite: if its
parts had sizes $a<b$ (they differ since $|K|$ is odd, and $a\ge1$ since $K$ has an edge), then
$K\setminus w$ for $w$ in the smaller part would have parts of sizes $a-1$ and $b$ with $b-(a-1)\ge2$,
hence no perfect matching. So $K$ contains an odd cycle; in particular $G$ has an odd cycle,
$|K|\ge\ell=2k+1$, and $\nu(\mathcal Q)\ge(|K|-1)/2\ge k$. By part~(i),
$q-1\ge\nu(\mathcal Q)\ge k$. Moreover, the $|A|$ pairs indexed by $A$ are disjoint from the
$|V(\mathcal Q)|$ pairs met by $\mathcal Q$, and $|V(\mathcal Q)|\ge2\nu(\mathcal Q)+1$ by
Lemma~\ref{lem:GE-allmissed}. Hence $(q-1-\nu(\mathcal Q))+2\nu(\mathcal Q)+1\le n$, that is,
$q\le n-\nu(\mathcal Q)\le n-k$. Consequently, if $G$ has no odd cycle, or $q\le k$, or
$q\ge n-k+1$, then $\cf^q(G)$ is pure.

Conversely, let $k+1\le q\le n-k$, and let $C$ be a shortest odd cycle of $G$; it is induced,
since a chord would split $C$ into two shorter cycles of which one is odd. By
Proposition~\ref{prop:vwc-smallfacet} with $\mu=k\ge1$, there is a facet with at most $n+q-1-k$
elements, while $\dim\cf^q(G)=n+q-2$ by Lemma~\ref{lem:vwc-facets}(1). Hence $\cf^q(G)$ is not
pure.
\end{proof}

For a whisker graph $G=W(H)$ the residual graph of $F=x_U\cup y_T$ is $H[Z]$ together with the
isolated vertices $y_j$, $j\in W$. We record this, and the fact that facets meet every pair.

\begin{lemma}\label{lem:facets}
Let $G=W(H)$ with $|V(H)|=n$, let $1\le q\le n$, and let $F=x_U\cup y_T\subseteq V(G)$ with
$A=U\cap T$, $Z=U\setminus T$ and $W=T\setminus U$. Then $G[F]$ is $H[U]$ with a pendant edge
$x_iy_i$ attached at each $i\in A$, together with the isolated vertices $y_j$, $j\in W$, and
$\nu(G[F])=|A|+\nu(H[Z]).$
Moreover, every facet $F$ of $\cf^q(G)$ satisfies $A\cup Z\cup W=[n]$, that is, $F$ contains a
vertex of every pair.
\end{lemma}
\begin{proof}
The formula is Lemma~\ref{lem:pend} applied to $\Gamma=H[U]$ and $B=A$. If $F$ is a face and
$i\notin U\cup T$, then $y_i$ is an isolated vertex of $G[F\cup\{y_i\}]$, so $F\cup\{y_i\}$ is a face
and $F$ is not a facet.
\end{proof}

\begin{corollary}\label{pure-range}
Let $G=W(H)$ be a whisker graph with $|V(H)|=n$, and let $1\le q\le n$. If $H$ has no odd cycle,
then $\cf^q(G)$ is pure. Otherwise, let $\ell=2k+1$ be the length of a shortest odd cycle of $H$.
Then $\cf^q(G)$ is pure if and only if
$ 1\le q<\Bigl\lceil\frac{\ell}{2}\Bigr\rceil$ or  $n-\Bigl\lfloor\frac{\ell}{2}\Bigr\rfloor<q\le n .$
\end{corollary}
\begin{proof}
The vertices $y_i$ have degree one in $G$, so every cycle of $G$ is a cycle of $H$. Since $G$ is
very well-covered, the claim is Theorem~\ref{thm:vwc-pure}.
\end{proof}

\section{Vertex decomposability of matching-free complexes}\label{shellable}
Throughout this section, let $H$ be a graph on $[n]$, and let
$G=W(H)$ be its whisker graph; as in the Conventions, $V(H)=\{x_1,\ldots,x_n\}$ is identified
with $[n]$, and the corresponding whisker vertices are $y_1,\ldots,y_n$. For $U\subseteq[n]$, we
write $x_U=\{x_i:i\in U\}$ and $y_U=\{y_i:i\in U\}$.
In this section, we show that $\cf^q(G)$ is vertex decomposable for
$q\le \lceil\girth(H)/2\rceil$ (Theorem~\ref{reg:main}). The proof
uses two families of auxiliary complexes. The first is the
\emph{cost complex} $\Theta(K,s)$ on the vertex set of a graph $K$.
The second is a family $\Omega(H,P,s)$ of links of
$\cf^{s+1}(G)$ that interpolates between $\cf^{s+1}(G)$ and
$\Theta(H,s)$.

\begin{definition}\label{def:cost-complex}
Let $K$ be a graph. For $T\subseteq V(K)$ put
$c_K(T)=|T|+\nu(K\setminus T),$
and for an integer $s$ define
$\Theta(K,s)=\{\,T\subseteq V(K)\;:\;c_K(T)\le s\,\}.$
\end{definition}

\begin{remark}\label{rem:cost-complex}
For $T\subseteq V(K)$ and $v\in V(K)\setminus T$ we have
$\nu\bigl(K\setminus(T\cup\{v\})\bigr)\le\nu(K\setminus T)\le\nu\bigl(K\setminus(T\cup\{v\})\bigr)+1$, and therefore
\begin{equation}\label{eq:c}
c_K(T)\;\le\;c_K(T\cup\{v\})\;\le\;c_K(T)+1 .
\end{equation}
In particular $\Theta(K,s)$ is a simplicial complex. Since $c_K(\emptyset)=\nu(K)$, the complex $\Theta(K,s)$ is nonvoid if and only if $\nu(K)\le s$.
\end{remark}

\begin{definition}\label{def:Omega}
For $P\subseteq[n]$ and an integer $s\ge0$ put
\[
\Omega(H,P,s)=\bigl\{\,F\subseteq V(G)\setminus x_P \;:\; \nu\bigl(G[F\cup x_P]\bigr)\le s\,\bigr\}.
\]
\end{definition}

\begin{remark}\label{rem:Omega}
We record three immediate observations.
\begin{enumerate}
\item If $x_P$ is a face of $\cf^{s+1}(G)$, then $\Omega(H,P,s)=\link_{\cf^{s+1}(G)}(x_P)$.
\item $\Omega(H,P,s)$ is nonvoid if and only if $\nu(H[P])\le s$, because $\nu(G[x_P])=\nu(H[P])$.
\item $\cf^{q}(G)=\Omega(H,\emptyset,q-1)$ for every $q\ge1$.
\end{enumerate}
\end{remark}

\begin{lemma}\label{lem:Omega-Theta}
The map $T\mapsto\{y_t: t\in T\}$ is an isomorphism $\Theta(H,s)\to\Omega(H,V(H),s)$.
\end{lemma}
\begin{proof}
For $P=V(H)$ we have $V(G)\setminus x_P=\{y_v: v\in V(H)\}$, so every face of $\Omega(H,V(H),s)$ is of the form $F_T=\{y_t:t\in T\}$ for a unique $T\subseteq V(H)$.
The graph $G[F_T\cup x_{V(H)}]$ is obtained from $H$ by attaching the pendant edge $\{x_t,y_t\}$ at each $t\in T$. By Lemma~\ref{lem:pend},
\[
\nu\bigl(G[F_T\cup x_{V(H)}]\bigr)=|T|+\nu(H\setminus T)=c_H(T).
\]
Hence $F_T\in\Omega(H,V(H),s)$ if and only if $T\in\Theta(H,s)$.
\end{proof}

\begin{lemma}\label{lem:xshed}
Let $j\in V(H)\setminus P$, let $\Omega=\Omega(H,P,s)$ be nonvoid, and let $G'=G\setminus\{x_j,y_j\}=W(H\setminus j)$.
\begin{enumerate}
\item $\link_\Omega(x_j)=\Omega(H,P\cup\{j\},s)$.
\item $\Omega\setminus x_j=y_j*\Omega(H\setminus j,P,s)$.
\item  If $\Omega$ is not a simplex, then no face of $\operatorname{link}_\Omega(x_j)$ is a facet of
$\Omega\setminus x_j$. In particular, if $\{x_j\}\in\Omega$, then $x_j$ satisfies condition~\textup{(2)}
of Definition~\ref{def:vd} for $\Omega$.
\end{enumerate}
\end{lemma}
\begin{proof}
(1) This is immediate from the definitions; if $\{x_j\}\notin\Omega$, both sides are void.

(2) Let $F\subseteq V(G)\setminus(x_P\cup\{x_j\})$. The only neighbor of $y_j$ in $G$ is $x_j\notin F\cup x_P$, so $y_j$ is isolated in $G[F\cup x_P]$. Hence
$\nu(G[F\cup x_P])=\nu\bigl(G'[(F\setminus y_j)\cup x_P]\bigr)$, which gives the claim. (The complex $\Omega(H\setminus j,P,s)$ is nonvoid since $\nu(H[P])\le s$.)

(3) By the remark after Definition~\ref{def:vd}, it suffices to show that no facet $F$ of $\link_\Omega(x_j)$ is a facet of $\Omega\setminus x_j$. Suppose such an $F$ exists. 
By (2), $y_j\in F$, so $F=F_0\cup\{y_j\}$ with $F_0\subseteq V(G')\setminus x_P$. Since $F\in\link_\Omega(x_j)$, we have $F\cup\{x_j\}\in\Omega$, i.e., $\nu\bigl(G[F\cup\{x_j\}\cup x_P]\bigr)\le s$. Put $\Gamma=G[F_0\cup\{x_j\}\cup x_P]$. Then $G[F\cup\{x_j\}\cup x_P]$ is obtained from $\Gamma$ by
attaching the pendant edge $x_jy_j$ at $x_j$, and $\Gamma\setminus x_j=G'[F_0\cup x_P]$. Hence
Lemma~\ref{lem:pend} (with $B=\{x_j\}$) gives $1+\nu(G'[F_0\cup x_P])\le s$.
Suppose some $v\in V(G')\setminus(x_P\cup F_0)$ exists. Adding a vertex increases the matching number by at most one, so
$\nu\bigl(G'[F_0\cup\{v\}\cup x_P]\bigr)\le\nu\bigl(G'[F_0\cup x_P]\bigr)+1\le s .$
By (2), $F_0\cup\{v,y_j\}\in\Omega\setminus x_j$, and it properly contains $F$. This contradicts the maximality of $F$ in $\Omega\setminus x_j$. Therefore $F_0=V(G')\setminus x_P$, and
$F\cup\{x_j\}=(V(G')\setminus x_P)\cup\{x_j,y_j\}=V(G)\setminus x_P$
is a face of $\Omega$. Since every face of $\Omega$ is a subset of $V(G)\setminus x_P$, the complex $\Omega$ is the full simplex on $V(G)\setminus x_P$, contradicting the hypothesis that $\Omega$ is not a simplex.
\end{proof}

\begin{definition}\label{def:Gm}
For $g\in\{3,4,5,\dots\}\cup\{\infty\}$ let $\mathcal G_g$ be the class of finite graphs $K$ with $\girth(K)\ge g$, that is, all of whose cycles have length at least $g$. For $g=\infty$ this is the class of finite forests. Put
\[
\kappa(g)=\Bigl\lceil \frac{g}{2}\Bigr\rceil-1 ~ (g<\infty),~ \kappa(\infty)=\infty .
\]
\end{definition}

\begin{remark}\label{rem:Gm}
Note the following.
\begin{enumerate}
\item $\mathcal G_g$ is closed under induced subgraphs, in particular under $K\mapsto K\setminus v$.
\item If $\girth(H)=g$, then $H\in\mathcal G_{g'}$ for every $g'\le g$.
\end{enumerate}
\end{remark}

\begin{lemma}\label{lem:mindeg2}
Let $K\in\mathcal G_g$ be nonempty with minimum degree $\delta(K)\ge2$, and let $s\le\kappa(g)$. Then $\Theta(K,s)$ is void or equal to $\{\emptyset\}$.
\end{lemma}
\begin{proof}
Since $\delta(K)\ge2$, $K$ contains a cycle, so $g<\infty$. Assume $\nu(K)\le s$, since otherwise $\Theta(K,s)$ is void. A cycle of length $c\ge g$ in $K$ gives
\[
\lfloor c/2\rfloor\le\nu(K)\le s\le\lceil g/2\rceil-1 .
\]
For $g$ even this is impossible, since $\lfloor c/2\rfloor\ge g/2$. So $g=2r+1$ with $r\ge1$, and then $c=g$. Thus every cycle of $K$ has length exactly $g$, and $\nu(K)=s=r$. Every cycle $C$ of $K$ is induced, since a chord would create a cycle of length less than $g$.

Suppose some vertex $u\notin V(C)$ lies in the component of $C$. Then we may choose such a $u$ adjacent to some $c\in V(C)$. The edge $uc$ together with a perfect matching of the path $C\setminus c$ (which has $2r$ vertices) is a matching of size $r+1>s$, a contradiction. So $C$ is a component of $K$. Any other component contains a cycle (as $\delta(K)\ge2$), hence has matching number at least $r$, which would give $\nu(K)\ge2r>r$. Therefore $K=C_{2r+1}$.

Let $\emptyset\ne T\subseteq V(K)$ and $t=|T|$. Then $K\setminus T$ is a disjoint union of at most $t$ paths with $2r+1-t$ vertices in total. A path on $\ell$ vertices has matching number $\lfloor\ell/2\rfloor\ge(\ell-1)/2$, so
\[
\nu(K\setminus T)\ \ge\ \frac{(2r+1-t)-t}{2}=r-t+\tfrac12 ,
\]
and, $\nu$ being an integer, $\nu(K\setminus T)\ge r-t+1$. Hence $c_K(T)\ge r+1>s$, so $T\notin\Theta(K,s)$. Since $c_K(\emptyset)=\nu(K)=r=s$, we get $\Theta(K,s)=\{\emptyset\}$.
\end{proof}

\begin{proposition}\label{prop:Theta-VD}
Let $K\in\mathcal G_g$ and $s\le\kappa(g)$ with $\nu(K)\le s$. Then $\Theta(K,s)$ is vertex decomposable.
\end{proposition}
\begin{proof}
We use induction on $|V(K)|$. Write $\Theta=\Theta(K,s)$; it is nonvoid since $\nu(K)\le s$. We may assume that $\Theta$ is not a simplex; in particular $\Theta\ne\{\emptyset\}$, since $\{\emptyset\}$ is a simplex. Then $K$ is nonempty, and by Lemma~\ref{lem:mindeg2} it has a vertex of degree at most $1$. In every induction step below the new graph lies in $\mathcal G_g$ (Remark~\ref{rem:Gm}), its parameter is at most $s\le\kappa(g)$, and its matching number is at most that parameter; this is checked where each complex is introduced.

\smallskip\noindent
\emph{Case 1: $K$ has an isolated vertex $p$.} For $T\not\ni p$ we have $c_K(T)=c_{K\setminus p}(T)$ and $c_K(T\cup\{p\})=c_{K\setminus p}(T)+1$.
If $\{p\}\notin\Theta$, then $\Theta=\Theta(K\setminus p,s)$ and $\nu(K\setminus p)=\nu(K)\le s$, so $\Theta$ is vertex decomposable by induction.
Otherwise $\{p\}\in\Theta$, so $1+\nu(K)\le s$, and
\[
\Theta\setminus p=\Theta(K\setminus p,s),~ \link_\Theta(p)=\Theta(K\setminus p,s-1)
\]
are nonvoid and vertex decomposable by induction.
Let $T$ be a facet of $\link_\Theta(p)$. If $T\ne V(K)\setminus\{p\}$, pick $v\notin T\cup\{p\}$. By \eqref{eq:c}, $c_{K\setminus p}(T\cup v)\le(s-1)+1=s$, so $T\cup v\in\Theta\setminus p$ and $T$ is not a facet of $\Theta\setminus p$. If $T=V(K)\setminus\{p\}$, then $V(K)=T\cup\{p\}\in\Theta$ and $\Theta$ is a simplex, which was excluded. Hence $p$ is a shedding vertex and $\Theta$ is vertex decomposable.

\smallskip\noindent
\emph{Case 2: $K$ has a leaf $p$ with neighbor $w$.} Put $K'=K\setminus\{p,w\}$. For $T_0\subseteq V(K')$, Lemma~\ref{lem:pend} applied to $\Gamma=K\setminus T_0\setminus p$ and $B=\{w\}$ gives $\nu(K\setminus T_0)=1+\nu(K'\setminus T_0)$. Also, $K\setminus T_0\setminus w$ is $K'\setminus T_0$ together with the isolated vertex $p$. Hence
\begin{equation}\label{eq:cone}
c_K(T_0)=c_K(T_0\cup\{w\})=c_{K'}(T_0)+1 .
\end{equation}
So the faces of $\Theta$ not containing $p$ are exactly the faces of the cone $w*\Theta(K',s-1)$. Moreover $\nu(K')\le s-1$, because $\nu(K)=c_K(\emptyset)=1+\nu(K')\le s$; thus $\Theta(K',s-1)$ is nonvoid.
If $\{p\}\notin\Theta$, then $\Theta=w*\Theta(K',s-1)$ is vertex decomposable by induction.

Now let $\{p\}\in\Theta$. Then
\[
\Theta\setminus p=w*\Theta(K',s-1),~ \link_\Theta(p)=\Theta(K\setminus p,s-1),
\]
where the latter is nonvoid since $1+\nu(K\setminus p)\le s$. Both are vertex decomposable by induction. Let $T$ be a facet of $\link_\Theta(p)$.
If $w\notin T$, then $T\in\Theta\setminus p$, hence $T\cup\{w\}\in\Theta\setminus p$ by the cone structure, so $T$ is not a facet of $\Theta\setminus p$.
If $w\in T$, put $T_0=T\setminus\{w\}$. From $T\cup\{p\}\in\Theta$ we get $|T_0|+2+\nu(K'\setminus T_0)\le s$, i.e., $c_{K'}(T_0)\le s-2$.
If $T$ were a facet of $\Theta\setminus p$, then $T_0$ would be a facet of $\Theta(K',s-1)$. But for every $v\in V(K')\setminus T_0$, \eqref{eq:c} gives $c_{K'}(T_0\cup v)\le s-1$, i.e., $T_0\cup v\in\Theta(K',s-1)$. So no such $v$ exists, $T_0=V(K')$, and $T\cup\{p\}=V(K)\in\Theta$, so $\Theta$ is a simplex, which was excluded. Hence $p$ is a shedding vertex and $\Theta$ is vertex decomposable.
\end{proof}

\begin{theorem}\label{thm:Omega-VD}
Let $H\in\mathcal G_g$, $P\subseteq V(H)$ and $0\le s\le\kappa(g)$ with $\nu(H[P])\le s$. Then $\Omega(H,P,s)$ is vertex decomposable.
\end{theorem}
\begin{proof}
We use induction on $|V(H)\setminus P|$. Here $H$ is allowed to be a graph on any finite set of indices, not necessarily $[n]$, and $G=W(H)$; this is needed because the induction passes from $H$ to $H\setminus j$. If $P=V(H)$, then $\Omega(H,P,s)\cong\Theta(H,s)$ by Lemma~\ref{lem:Omega-Theta}, and $\nu(H)=\nu(H[P])\le s$, so Proposition~\ref{prop:Theta-VD} applies.

Let $P\ne V(H)$, pick $j\in V(H)\setminus P$, and put $\Omega=\Omega(H,P,s)$; it is nonvoid since $\nu(H[P])\le s$. We may assume $\Omega$ is not a simplex. By Lemma~\ref{lem:xshed}(2),
\[
\Omega\setminus x_j=y_j*\Omega(H\setminus j,P,s).
\]
Here $H\setminus j\in\mathcal G_g$, $\nu\bigl((H\setminus j)[P]\bigr)=\nu(H[P])\le s$, and $|V(H\setminus j)\setminus P|<|V(H)\setminus P|$. So $\Omega(H\setminus j,P,s)$ is vertex decomposable by induction, and hence so is the cone $\Omega\setminus x_j$.

If $\{x_j\}\notin\Omega$, then no face of $\Omega$ contains $x_j$, so $\Omega=\Omega\setminus x_j$ is vertex decomposable.

If $\{x_j\}\in\Omega$, then $\nu(H[P\cup\{j\}])=\nu(G[x_P\cup\{x_j\}])\le s$. By Lemma~\ref{lem:xshed}(1), $\link_\Omega(x_j)=\Omega(H,P\cup\{j\},s)$, which is vertex decomposable by induction, since $|V(H)\setminus(P\cup\{j\})|<|V(H)\setminus P|$. As $\Omega\setminus x_j$ is also vertex decomposable and, by Lemma~\ref{lem:xshed}(3), $x_j$ satisfies condition~\textup{(2)} of Definition~\ref{def:vd}, the vertex $x_j$ is a shedding vertex and $\Omega$ is vertex decomposable.
\end{proof}

We are now ready to prove the main result of this section.

\begin{theorem}\label{reg:main}
Let $H$ be a graph on $[n]$ and $G=W(H)$. Then $\cf^q(G)$ is vertex decomposable, and hence shellable, for all
\[
1\le q\le
\begin{cases}
\lceil \girth(H)/2\rceil, & \text{if } \girth(H)<\infty,\\[2pt]
n=\nu(G), & \text{if } \girth(H)=\infty .
\end{cases}
\]
\end{theorem}
\begin{proof}
Put $g=\girth(H)$ and $s=q-1$. The hypothesis on $q$ says exactly that $0\le s\le\kappa(g)$, and $H\in\mathcal G_g$ by definition of the girth. Since $\cf^q(G)=\Omega(H,\emptyset,s)$ by Remark~\ref{rem:Omega}(3) and $\nu(H[\emptyset])=0\le s$, Theorem~\ref{thm:Omega-VD} shows that $\cf^q(G)$ is vertex decomposable. Hence it is shellable by \cite[Theorem~11.3]{BW97}.
\end{proof}

\begin{corollary}\label{main-cor}
Let $H$ be a graph on $[n]$, $G=W(H)$, and $1\le q\le n$.
\begin{enumerate}
\item If $q\le\lceil\girth(H)/2\rceil$ (any $q\le n$ if $H$ is a forest), then $R/I(G)^{[q]}$ is sequentially Cohen--Macaulay over every field.
\item If $q\le\lfloor\girth(H)/2\rfloor$ (any $q\le n$ if $H$ is a forest), then $R/I(G)^{[q]}$ is Cohen--Macaulay over every field.
\item If $\girth(H)=2k+1$ is odd and $q=k+1$, then $\cf^q(G)$ is shellable but not pure, so $R/I(G)^{[q]}$ is sequentially Cohen--Macaulay but not Cohen--Macaulay.
\end{enumerate}
\end{corollary}
\begin{proof}
(1) Shellable complexes are sequentially Cohen--Macaulay over every field, so this follows from Theorem~\ref{reg:main}.

(2) By Theorem~\ref{reg:main}, $\cf^q(G)$ is shellable. It is also pure by Corollary~\ref{pure-range}: if $H$ has an odd cycle, let $\ell=2k'+1$ be the length of a shortest one. Then $\ell\ge\girth(H)$, so $q\le\lfloor\girth(H)/2\rfloor\le k'$, and if $H$ has no odd cycle, $\cf^q(G)$ is pure for all $q$. A pure shellable complex is Cohen--Macaulay.

(3) Here $\girth(H)=2k+1$ is the length of a shortest odd cycle, $k'=k$ and $q=k+1$. Since $H$ has a cycle of length $2k+1$, we have $n\ge2k+1$, so $q=k+1\le k'$ and $q\ge n-k'+1$ both fail. By Corollary~\ref{pure-range}, $\cf^q(G)$ is not pure. It is shellable by Theorem~\ref{reg:main}, since $q=\lceil\girth(H)/2\rceil$.
\end{proof}

For forests, the Cohen--Macaulay statement in (2) was obtained earlier by Das, Roy and Saha
\cite{DRS24}; the range in (1) and (2) for graphs with cycles, and part (3), are new.

The bound in Theorem~\ref{reg:main} is sharp for every graph $H$ of girth $g$ with $4\le g<\infty$: by Corollary~\ref{cor:scm-trianglefree} below, $\cf^q(W(H))$ is then not even sequentially Cohen--Macaulay for $q=\lceil g/2\rceil+1$. For girth $3$ the bound need not be sharp; for instance, $\cf^q(W(K_3))$ is vertex decomposable for every $1\le q\le3$. The following example illustrates the sharpness for even cycles.

\begin{example}\label{sharp-bd}
Let $m\ge 2$, $H=C_{2m}$ and $q=m+1=\lceil 2m/2\rceil+1$. Since $\nu(C_{2m})=m=q-1$, the set
$x_{[2m]}$ is a face of $\cf^{q}(G)$, and its link is $\{y_T : c_{C_{2m}}(T)\le m\}$
(Lemmas~\ref{lem:pend} and~\ref{lem:Omega-Theta}). Let $V_1$ and $V_0$ be the sets of odd and even vertices of $C_{2m}$.
If $T\subseteq V_1$ with $|T|=t\ge1$, then $C_{2m}\setminus T$ is a disjoint union of $t$ paths with
an odd number of vertices each, $2m-t$ in total, so $\nu(C_{2m}\setminus T)=m-t$ and
$c(T)=m$. The same holds for $T\subseteq V_0$. If $i\in V_1$ and $j\in V_0$, then $C_{2m}\setminus\{i,j\}$
is a union of at most two paths, each with an even number of vertices, so $c(\{i,j\})=2+(m-1)=m+1$; by
\eqref{eq:c} no face contains both an odd and an even vertex. Hence
\[
\operatorname{link}_{\cf^{m+1}(W(C_{2m}))}(x_{[2m]})\cong 2^{V_0}\sqcup 2^{V_1},
\]
a pure disconnected complex of dimension $m-1\ge 1$. Therefore $\cf^{m+1}(W(C_{2m}))$
is neither shellable nor sequentially Cohen--Macaulay (Lemmas~\ref{lem:scm-conn} and~\ref{lem:scm-links}), and the bound of Theorem~\ref{reg:main} is sharp
for every even cycle. Proposition~\ref{prop:evencycle} below generalizes this example to every
induced even cycle and to a whole range of $q$.

\end{example}

\section{Reduction to whisker graphs}\label{sec:rigid}

In this section we show that, among all very well-covered graphs, the whisker graphs
are the only ones for which some squarefree power $I(G)^{[q]}$ with $2\le q\le n-1$ is
sequentially Cohen--Macaulay (Theorem~\ref{thm:rigid} for Cohen--Macaulay graphs, and
Theorem~\ref{thm:noncm-vwc} and Corollary~\ref{cor:all-vwc} in general). This reduces the study of
the Cohen--Macaulay property of squarefree powers of very well-covered graphs to the case of whisker
graphs, which is treated in Section~\ref{shellable} and in the remaining sections. Throughout, we write $N(v)=N_G(\{v\})$ and
$N[v]=N_G[\{v\}]$.

\begin{lemma}\label{lem:two-edges}
Let $G$ be a graph and $\{a,b\}\in E(G)$. Suppose that there are distinct vertices
$u_1,u_2\in N(a)\setminus N[b]$ and distinct vertices $w_1,w_2\in N(b)\setminus N[a]$ such that
$\{u_1,u_2\}\notin E(G)$, $\{w_1,w_2\}\notin E(G)$, and
\begin{equation}\label{eq:dom}
N(u_1)\cup N(u_2)\cup N(w_1)\cup N(w_2)\subseteq N[a]\cup N[b].
\end{equation}
Then there is a face $\sigma=\{a,b\}\cup\tau$ of $\cf^2(G)$, where $\tau$ is an independent set of $G$
contained in $V(G)\setminus(N[a]\cup N[b])$, such that $\link_{\cf^2(G)}(\sigma)$ has
$1$-dimensional faces in two different connected components. In particular, $\cf^2(G)$ is not
sequentially Cohen--Macaulay.
\end{lemma}
\begin{proof}
Put $V'=V(G)\setminus(N[a]\cup N[b])$, let $\tau$ be a maximal independent set of $G[V']$
\textup{(}possibly empty\textup{)}, and let $\sigma=\{a,b\}\cup\tau$. The only edge of
$G[\sigma]$ is $ab$, so $\sigma\in\cf^2(G)$. Let $\Lambda=\link_{\cf^2(G)}(\sigma)$. A set
$L\subseteq V(G)\setminus\sigma$ is a face of $\Lambda$ if and only if any two edges of
$G[L\cup\sigma]$ meet. Put $P_1=N(a)\setminus N[b]$, $P_2=N(b)\setminus N[a]$ and
$P_{12}=N(a)\cap N(b)$.

If $v\in V'\setminus\tau$, then $v$ has a neighbor $t\in\tau$ by the maximality of $\tau$, and
the edges $vt$ and $ab$ are disjoint. Hence every vertex of $\Lambda$ lies in
$P_1\cup P_2\cup P_{12}$. If $u\in P_1$ and $w\in P_2$, then the edges $ua$ and $wb$ are
disjoint, so no face of $\Lambda$ meets both $P_1$ and $P_2$. If $c\in P_{12}$ and $u\ne c$ is
another vertex of $\Lambda$, then $u$ is adjacent to $a$ or to $b$, say to $a$, and the edges
$ua$ and $cb$ are disjoint; so no face of $\Lambda$ with more than one element contains $c$.
Therefore the $1$-dimensional faces of $\Lambda$ contained in $P_1$ and those contained in $P_2$
lie in different connected components of $\Lambda$.

By~\eqref{eq:dom}, $u_1$ and $u_2$ have no neighbor in $\tau\subseteq V'$; they are adjacent
neither to $b$ nor to each other. So every edge of $G[\sigma\cup\{u_1,u_2\}]$ contains $a$, and
$\{u_1,u_2\}\in\Lambda$. In the same way $\{w_1,w_2\}\in\Lambda$. By Lemma~\ref{lem:scm-conn},
$\Lambda$ is not sequentially Cohen--Macaulay, and since links of sequentially Cohen--Macaulay
complexes are sequentially Cohen--Macaulay (Lemma~\ref{lem:scm-links}), neither is $\cf^2(G)$.
\end{proof}

\begin{lemma}\label{lem:config}
Let $G$ be a Cohen--Macaulay very well-covered graph, labeled as in Theorem~\ref{thm:vwc} so that
\textup{(V1)}--\textup{(V4)} hold. Then $G$ is a whisker graph if and only if neither of the
following holds:
\begin{enumerate}
\item[\textup{(A)}] there are distinct $i,j,k$ with $\{x_i,y_j\},\{x_j,y_k\}\in E(G)$;
\item[\textup{(B)}] there are distinct $i,j,m$ with $\{x_i,y_j\},\{x_j,x_m\}\in E(G)$.
\end{enumerate}
\end{lemma}
\begin{proof}
Assume that neither \textup{(A)} nor \textup{(B)} holds, and fix $j$. If $y_j$ has no neighbor
$x_i$ with $i\ne j$, then $y_j$ is a leaf. Otherwise $\{x_i,y_j\}\in E(G)$ for some $i\ne j$.
Then $x_j$ has no neighbor $y_k$ with $k\ne j$: we would have $k\ne i$ by \textup{(V4)}, and
\textup{(A)} would hold. Also $x_j$ has no neighbor $x_m$: we would have $m\ne i$ by
\textup{(V3)}, and \textup{(B)} would hold. So $x_j$ is a leaf. Thus every pair contains a leaf
$\ell_j$ of $G$; let $p_j$ be the other vertex of the $j$th pair and $H=G[\{p_1,\dots,p_n\}]$.
Then $G=W(H)$.

Conversely, a whisker graph has a unique perfect matching, which consists of pendant edges. If
\textup{(A)} or \textup{(B)} holds, then neither $x_j$ nor $y_j$ is a leaf, so the perfect
matching $\{\{x_i,y_i\}:i\in[n]\}$ of $G$ does not consist of pendant edges, and $G$ is not a
whisker graph.
\end{proof}

\begin{theorem}\label{thm:rigid}
Let $G$ be a Cohen--Macaulay very well-covered graph on $2n$ vertices, $n\ge3$. The following
are equivalent:
\begin{enumerate}
\item[\textup{(a)}] $\cf^q(G)$ is sequentially Cohen--Macaulay over some field, for some
$2\le q\le n-1$;
\item[\textup{(b)}] $\cf^2(G)$ is sequentially Cohen--Macaulay over every field;
\item[\textup{(c)}] $G$ is a whisker graph.
\end{enumerate}
\end{theorem}
\begin{proof}
\textup{(c)}$\Rightarrow$\textup{(b)}: if $G=W(H)$, then $\cf^2(G)$ is vertex decomposable by
Theorem~\ref{reg:main}, since $2\le\lceil\girth(H)/2\rceil$ if $H$ has a cycle and $2\le n$ if $H$ is a
forest. \textup{(b)}$\Rightarrow$\textup{(a)} is trivial.

\textup{(a)}$\Rightarrow$\textup{(c)}: label $G$ so that \textup{(V1)}--\textup{(V4)} hold, and
suppose that $G$ is not a whisker graph. By Lemma~\ref{lem:config}, one of the configurations
\textup{(A)}, \textup{(B)} occurs, on a set $\Sigma$ of three indices. We show that $\cf^q(G)$ is
not sequentially Cohen--Macaulay for any $2\le q\le n-1$.

\emph{Step 1: $q=2$.} We verify the hypotheses of Lemma~\ref{lem:two-edges}.

\emph{Case \textup{(A)}.} Applying \textup{(V2)} to $\{y_k,x_j\}$ and $\{y_j,x_i\}$ we get
$\{x_i,y_k\}\in E(G)$. Put $a=x_i$, $b=y_k$, $u_1=y_i$, $u_2=y_j$, $w_1=x_j$ and $w_2=x_k$.
Then $u_1,u_2\in N(a)$, and they are adjacent neither to $b$ nor to each other by
\textup{(V1)}. Also $w_1,w_2\in N(b)$, and $\{x_i,x_j\}$, $\{x_i,x_k\}$, $\{x_j,x_k\}$ are not
edges, by \textup{(V3)} applied to $\{x_i,y_j\}$, $\{x_i,y_k\}$, $\{x_j,y_k\}$. For
\eqref{eq:dom}: if $\{x_p,y_i\}\in E(G)$, then \textup{(V2)} applied to $\{y_k,x_i\}$ and
$\{y_i,x_p\}$ gives $x_p\in N[b]$; if $\{x_p,y_j\}\in E(G)$, then \textup{(V2)} applied to
$\{y_k,x_j\}$ and $\{y_j,x_p\}$ gives $x_p\in N[b]$; if $\{z,x_j\}\in E(G)$, then \textup{(V2)}
applied to $\{z,x_j\}$ and $\{y_j,x_i\}$ gives $z\in N[a]$; and if $\{z,x_k\}\in E(G)$, then
\textup{(V2)} applied to $\{z,x_k\}$ and $\{y_k,x_i\}$ gives $z\in N[a]$.

\emph{Case \textup{(B)}.} Applying \textup{(V2)} to $\{x_m,x_j\}$ and $\{y_j,x_i\}$ we get
$\{x_i,x_m\}\in E(G)$. Put $a=x_i$, $b=x_m$, $u_1=y_i$, $u_2=y_j$, $w_1=x_j$ and $w_2=y_m$.
Then $u_1,u_2\in N(a)$ are not adjacent to each other, and they are not adjacent to $b$, since
by \textup{(V3)} an edge $\{x_m,y_i\}$ would exclude $\{x_m,x_i\}$ and an edge $\{x_m,y_j\}$
would exclude $\{x_m,x_j\}$. Also $w_1,w_2\in N(b)$; they are not adjacent to $a$, by
\textup{(V3)} applied to $\{x_i,y_j\}$ and because an edge $\{x_i,y_m\}$ would exclude
$\{x_i,x_m\}$; and they are not adjacent to each other, because an edge $\{x_j,y_m\}$ would
exclude $\{x_j,x_m\}$. For \eqref{eq:dom}: if $\{x_p,y_i\}\in E(G)$, then \textup{(V2)} applied
to $\{x_m,x_i\}$ and $\{y_i,x_p\}$ gives $x_p\in N[b]$; if $\{x_p,y_j\}\in E(G)$, then
\textup{(V2)} applied to $\{x_m,x_j\}$ and $\{y_j,x_p\}$ gives $x_p\in N[b]$; if
$\{z,x_j\}\in E(G)$, then \textup{(V2)} applied to $\{z,x_j\}$ and $\{y_j,x_i\}$ gives
$z\in N[a]$; and if $\{x_p,y_m\}\in E(G)$, then \textup{(V2)} applied to $\{x_i,x_m\}$ and
$\{y_m,x_p\}$ gives $x_p\in N[a]$.

In both cases $\cf^2(G)$ is not sequentially Cohen--Macaulay by Lemma~\ref{lem:two-edges}.

\emph{Step 2: $2\le q\le n-1$.} Choose $q-2$ indices outside $\Sigma$, which is possible since
$q-2\le n-3$, and let $\rho$ be the union of the corresponding pairs. Applying
Corollary~\ref{cor:pair-link} $q-2$ times we get $\link_{\cf^q(G)}(\rho)=\cf^2(G')$, where
$G'=G\setminus\rho$. The graph $G'$ satisfies \textup{(V1)}--\textup{(V4)} and contains the
same configuration on $\Sigma$, so $\cf^2(G')$ is not sequentially Cohen--Macaulay by Step~1.
Since links of sequentially Cohen--Macaulay complexes are sequentially
Cohen--Macaulay (Lemma~\ref{lem:scm-links}), $\cf^q(G)$ is not sequentially Cohen--Macaulay.
\end{proof}

\begin{remark}\label{rem:noV4}
The proof of \textup{(a)}$\Rightarrow$\textup{(c)} uses \textup{(V4)} only through
Lemma~\ref{lem:config}. Thus, if a very well-covered graph $G$ on $2n$ vertices, labeled as in
Theorem~\ref{thm:vwc}, contains a configuration \textup{(A)} or \textup{(B)}, then $\cf^q(G)$ is
not sequentially Cohen--Macaulay for any $2\le q\le n-1$.
\end{remark}

\begin{example}\label{ex:rigid}
By Lemma~\ref{lem:config}, a Cohen--Macaulay very well-covered graph that is not a whisker graph
has at least three pairs. Let $G$ have the edges $x_iy_i$ \textup{(}$i=1,2,3$\textup{)} and
$x_1y_2$, $x_1y_3$, $x_2y_3$. Then \textup{(V1)}--\textup{(V4)} hold, and $G$ contains the configuration \textup{(A)} with $(i,j,k)=(1,2,3)$; the edges $x_1y_2$, $x_1y_3$, $x_2y_3$ are all the
possible edges $x_iy_j$ with $i<j$.
In the proof of Theorem~\ref{thm:rigid} we take $a=x_1$ and $b=y_3$. Here $N[a]\cup N[b]=V(G)$, so
$\tau=\emptyset$ and $\sigma=\{x_1,y_3\}$. Every edge of $G[\{x_1,y_3,y_1,y_2\}]$ contains $x_1$, and every edge of
$G[\{x_1,y_3,x_2,x_3\}]$ contains $y_3$, while, for instance, the edges $x_1y_1$ and $x_2y_3$
are disjoint. Hence
\[
\link_{\cf^2(G)}(\{x_1,y_3\})=\bigl\langle\{y_1,y_2\},\ \{x_2,x_3\}\bigr\rangle
\]
is the disjoint union of two edges, and $I(G)^{[2]}$ is not sequentially Cohen--Macaulay.
\end{example}

\begin{corollary}\label{cor:vwc-range}
Let $G$ be a Cohen--Macaulay very well-covered graph on $2n$ vertices, $n\ge3$, which is not a
whisker graph, and let $1\le q\le n$. Then the following are equivalent:
\begin{enumerate}
\item $R/I(G)^{[q]}$ is Cohen--Macaulay;
\item $R/I(G)^{[q]}$ is sequentially Cohen--Macaulay;
\item $q\in\{1,n\}$.
\end{enumerate}
\end{corollary}
\begin{proof}
For $q=1$ the ring is Cohen--Macaulay by hypothesis, and for $q=n=\nu(G)$ by
Remark~\ref{rem:top-power}. The values $2\le q\le n-1$ are
excluded by Theorem~\ref{thm:rigid}.
\end{proof}

We now show that Theorem~\ref{thm:rigid} remains true without the Cohen--Macaulay assumption: a
very well-covered graph which is not Cohen--Macaulay has no sequentially Cohen--Macaulay squarefree
power $I(G)^{[q]}$ with $2\le q\le n-1$. For a very well-covered graph $G$ labeled as in
Theorem~\ref{thm:vwc} we write $i\to j$ if $i\ne j$ and $\{x_i,y_j\}\in E(G)$.

\begin{lemma}\label{lem:noAB}
Let $G$ be a very well-covered graph on $2n$ vertices, labeled so that \textup{(V1)}--\textup{(V3)}
hold, and suppose that neither of the configurations \textup{(A)}, \textup{(B)} of
Lemma~\ref{lem:config} occurs. Then one of the following holds:
\begin{enumerate}
\item $G$ is Cohen--Macaulay;
\item there are $i\ne j$ with $i\to j$ and $j\to i$, and $G[\{x_i,y_i,x_j,y_j\}]$ is a
connected component of $G$ isomorphic to the $4$-cycle $C_4$.
\end{enumerate}
\end{lemma}
\begin{proof}
A directed path $i\to j\to k$ with $i,j,k$ pairwise distinct is exactly a configuration
\textup{(A)}. Hence the relation $\to$ has no directed cycle of length at least $3$.

Suppose first that there are no $i\ne j$ with $i\to j$ and $j\to i$. Then $\to$ has no directed
cycle, so the indices can be renumbered so that $i\to j$ implies $i<j$. Renumbering the pairs
preserves \textup{(V1)}--\textup{(V3)}, and the new labeling satisfies \textup{(V4)}. So $G$
is Cohen--Macaulay by Theorem~\ref{thm:vwc}.

Now let $i\to j$ and $j\to i$, that is, $\{x_i,y_j\},\{x_j,y_i\}\in E(G)$. By \textup{(V3)},
$\{x_i,x_j\}\notin E(G)$. Let $k\notin\{i,j\}$. If $\{x_i,y_k\}\in E(G)$, then
$j\to i\to k$ is a configuration \textup{(A)}. If $\{x_k,y_i\}\in E(G)$, then $k\to i\to j$ is a
configuration \textup{(A)}. If $\{x_i,x_k\}\in E(G)$, then $\{x_j,y_i\},\{x_i,x_k\}\in E(G)$ is a
configuration \textup{(B)} on the indices $j,i,k$. Hence $N(x_i)=\{y_i,y_j\}$ and, since $\{y_1,\dots,y_n\}$ is
independent by \textup{(V1)}, $N(y_i)=\{x_i,x_j\}$. By symmetry $N(x_j)=\{y_i,y_j\}$ and
$N(y_j)=\{x_i,x_j\}$. Thus $\{x_i,y_i,x_j,y_j\}$ induces the $4$-cycle
$x_iy_ix_jy_jx_i$ and has no neighbors outside.
\end{proof}

\begin{lemma}\label{lem:C4-comp}
Let $G$ be a very well-covered graph on $2n$ vertices, $n\ge3$, which has a connected component
isomorphic to $C_4$, and let $2\le q\le n-1$. Then $\cf^q(G)$ is not sequentially Cohen--Macaulay
over any field.
\end{lemma}
\begin{proof}
Write $G=C\sqcup K$, where $C$ is the $4$-cycle $x_1y_1x_2y_2x_1$ and $K=G\setminus V(C)$. The
graph $K$ is very well-covered with the induced labeling and has $m=n-2\ge1$ pairs; let $X_K$
and $Y_K$ be its sets of $x$- and $y$-vertices. For $S\subseteq X_K$ we have
$\nu(K[Y_K\cup S])=|S|$: the edges $\{x_l,y_l\}$, $x_l\in S$, form a matching, and since $Y_K$ is
independent every edge of $K[Y_K\cup S]$ contains a vertex of $S$.

Since $m\ge q-1$, we can choose $U_0\subseteq X_K$ with $|U_0|=q-2$; put
$X'=X_K\setminus U_0$, so $|X'|=m-q+2\ge1$, and $\sigma=Y_K\cup U_0$. Then
$\nu(G[\sigma])=q-2$, so $\sigma\in\cf^q(G)$. A set $F=D\cup U$ with $D\subseteq V(C)$ and
$U\subseteq X'$ satisfies
\[
\nu(G[\sigma\cup F])=\nu(C[D])+|U_0|+|U|,
\]
so $F\in\Lambda:=\link_{\cf^q(G)}(\sigma)$ if and only if $\nu(C[D])+|U|\le1$. Now
$\nu(C[D])\le1$ if and only if $D\ne V(C)$, and $\nu(C[D])=0$ if and only if $D$ is contained in
$\{x_1,x_2\}$ or in $\{y_1,y_2\}$. Hence $\Lambda=\Lambda_1\cup\Lambda_2$, where $\Lambda_1$ is the boundary of the
simplex on $V(C)$, a $2$-sphere, $\Lambda_2=X'_\bullet*\Lambda_0$, where $X'_\bullet=\{\emptyset\}\cup\{\{x\}:x\in X'\}$ is the set $X'$ regarded as a $0$-dimensional complex, and $\Lambda_0$ is the complex with the two disjoint facets
$\{x_1,x_2\}$ and $\{y_1,y_2\}$; moreover $\Lambda_1\cap\Lambda_2=\Lambda_0$. Every facet of $\Lambda$ has three
elements, so $\Lambda$ is pure of dimension $2$. The complex $\Lambda_2$ is the join of a nonempty $0$-dimensional
complex with a nonempty complex, hence connected, and $\widetilde H_0(\Lambda_0;\K)\cong\K$. The
Mayer--Vietoris sequence
\[
\widetilde H_1(\Lambda;\K)\longrightarrow \widetilde H_0(\Lambda_0;\K)\longrightarrow
\widetilde H_0(\Lambda_1;\K)\oplus\widetilde H_0(\Lambda_2;\K)=0
\]
shows that $\widetilde H_1(\Lambda;\K)\ne0$. Thus $\Lambda$ is a pure $2$-dimensional complex
which is not Cohen--Macaulay, hence not sequentially Cohen--Macaulay. By
Lemma~\ref{lem:scm-links}, $\cf^q(G)$ is not sequentially Cohen--Macaulay.
\end{proof}

\begin{theorem}\label{thm:noncm-vwc}
Let $G$ be a very well-covered graph on $2n$ vertices which is not Cohen--Macaulay. Then
$R/I(G)^{[q]}$ is sequentially Cohen--Macaulay \textup{(}over some field\textup{)} if and only if
$q=n$; in that case it is Cohen--Macaulay.
\end{theorem}
\begin{proof}
For $q=1$: a very well-covered graph is well-covered, so $\cf^1(G)$ is pure, and a pure
sequentially Cohen--Macaulay complex is Cohen--Macaulay. By Theorem~\ref{thm:vwc}, the
Cohen--Macaulay property of a very well-covered graph does not depend on the field; so $R/I(G)$
is not sequentially Cohen--Macaulay over any field. For $q=n$ the ring is Cohen--Macaulay by
Remark~\ref{rem:top-power}.

Let $2\le q\le n-1$; then $n\ge3$. Label $G$ as in Theorem~\ref{thm:vwc}. If a configuration
\textup{(A)} or \textup{(B)} occurs, then $\cf^q(G)$ is not sequentially Cohen--Macaulay by
Remark~\ref{rem:noV4}. Otherwise, since $G$ is not Cohen--Macaulay,
Lemma~\ref{lem:noAB} shows that $G$ has a connected component isomorphic to $C_4$, and
Lemma~\ref{lem:C4-comp} applies.
\end{proof}

Together with Theorem~\ref{thm:rigid} we obtain a characterization for all very well-covered graphs.

\begin{corollary}\label{cor:all-vwc}
Let $G$ be a very well-covered graph on $2n$ vertices, $n\ge3$. Then $\cf^q(G)$ is sequentially
Cohen--Macaulay \textup{(}over some field\textup{)} for some $2\le q\le n-1$ if and only if $G$ is
a whisker graph. In that case $\cf^2(G)$ is vertex decomposable.
\end{corollary}
\begin{proof}
If $G=W(H)$, then $\cf^2(G)$ is vertex decomposable by Theorem~\ref{reg:main}, as in the proof
of Theorem~\ref{thm:rigid}. Conversely, suppose that $G$ is not a whisker graph. If $G$ is
Cohen--Macaulay, then $\cf^q(G)$ is not sequentially Cohen--Macaulay for $2\le q\le n-1$ by
Theorem~\ref{thm:rigid}; otherwise this holds by Theorem~\ref{thm:noncm-vwc}.
\end{proof}

\section{Depth and the Cohen--Macaulay range}
\label{depth-sq}

In this section we study $\depth R/I(G)^{[q]}$. For whisker graphs we first give a formula for the
depth and the regularity that is valid for all $H$ and all $q$ (Theorem~\ref{thm:general-depth-reg});
it is expressed in terms of the homology of certain complexes $\Gamma_r$. We then derive explicit
values in several ranges of $q$, in terms of the girth and the circumference of $H$
(Theorem~\ref{depth-formula}), prove the conjecture of Das, Roy and Saha on whisker cycles, and
determine exactly for which $q$ the ring $R/I(W(H))^{[q]}$ is Cohen--Macaulay
(Theorem~\ref{thm:cm-range}); this completes the proof of Theorem~\ref{intro-main}. In the last
part of the section we determine the depth of the large squarefree powers of the Cohen--Macaulay
very well-covered graphs that are not whisker graphs (Theorem~\ref{thm:nonwhisker-depth}).
Throughout, the following lower bound is used.

\begin{lemma}[{\cite[Proposition~1.1]{EHHS24}}]\label{lem:gendeg}
Let $I\subsetneq S=\K[z_1,\dots,z_m]$ be a nonzero squarefree
monomial ideal whose minimal generators all have degree at least
$d\ge2$. Then
$\depth(S/I)\ge d-1.$
\end{lemma}

\subsection{The complexes $\Gamma_r(H)$}\label{subsec:gamma}

In this subsection $H$ is a graph on $[n]$.

\begin{definition}\label{def:Gamma}
For $D\subseteq[n]$, put
$\gamma_H(D)=|D|-\nu(H[D]).$
Note that $\gamma_H(D)$ is the minimum number of vertices and edges of $H[D]$
needed to cover $D$ (each edge of a matching covers two vertices at the
cost of one). In particular, $\gamma_H(\emptyset)=0$. For an integer $r\ge0$,
define
$\Gamma_r(H)=\{D\subseteq[n]:\gamma_H(D)\le r\}.$
\end{definition}

\begin{lemma}\label{lem:Gamma-dual}
For every $r\ge0$, $\Gamma_r(H)$ is a simplicial complex on $[n]$
(for $r=0$ it is $\{\emptyset\}$; for $r\ge1$ it contains all singletons).
Moreover, for $0\le s\le n-1$,
$\Theta(H,s)=\Gamma_{n-s-1}(H)^\vee.$
Equivalently,
$T\in\Theta(H,s)\iff[n]\setminus T\notin\Gamma_{n-s-1}(H).$
\end{lemma}

\begin{proof}
Let $D\subseteq[n]$ and $v\notin D$. Adding $v$ to $H[D]$ can increase the
matching number by at most one, since deleting $v$ from a maximum matching
of $H[D\cup\{v\}]$ removes at most one edge. Hence
$0\le \nu(H[D\cup\{v\}])-\nu(H[D])\le1,$
and therefore
$\gamma_H(D)\le \gamma_H(D\cup\{v\})\le \gamma_H(D)+1.$
Thus $\gamma_H$ is monotone under inclusion, so $\Gamma_r(H)$ is closed under
taking subsets. Since $\gamma_H(\emptyset)=0\le r$, we have
$\emptyset\in\Gamma_r(H)$, so $\Gamma_r(H)$ is a simplicial complex.
For a singleton, $\gamma_H(\{v\})=1$, so $\{v\}\in\Gamma_r(H)$ if and only if
$r\ge1$.

Now let $T\subseteq[n]$ and put $D=[n]\setminus T$. Then
\[
\begin{aligned}
|T|+\nu\bigl(H[[n]\setminus T]\bigr)
&=n-|D|+\nu(H[D])
=n-\gamma_H(D).
\end{aligned}
\]
Hence
\[
\begin{aligned}
T\in\Theta(H,s)
&\iff |T|+\nu\bigl(H[[n]\setminus T]\bigr)\le s
\iff \gamma_H(D)\ge n-s\\
&\iff \gamma_H(D)>n-s-1
\iff D\notin\Gamma_{n-s-1}(H).
\end{aligned}
\]
This is precisely the statement $T\in\Gamma_{n-s-1}(H)^\vee$. Therefore
$\Theta(H,s)=\Gamma_{n-s-1}(H)^\vee$.
\end{proof}

\begin{lemma}\label{lem:faces}
Every face of $\Gamma_r(H)$ has at most $2r$ elements. If $M$ is a matching of $H$ with $r$ edges, then $V(M)\in\Gamma_r(H)$; hence $\dim\Gamma_r(H)=2r-1$ whenever $r\le\nu(H)$.
\end{lemma}
\begin{proof}
Since $\nu(H[D])\le|D|/2$, we have $\gamma_H(D)\ge|D|/2$, so $\gamma_H(D)\le r$ forces $|D|\le2r$. If $M$ has $r$ edges, then $|V(M)|=2r$ and $\nu(H[V(M)])=r$, so $\gamma_H(V(M))=r$.
\end{proof}

\begin{lemma}\label{lem:cycle}
Let $r\ge1$ and suppose $H$ contains a cycle $C$ of length $\ell\ge2r+1$. Then $\widetilde H_{2r-1}(\Gamma_r(H);\K)\ne0$ for every field $\K$.
\end{lemma}
\begin{proof}
By Lemma~\ref{lem:faces}, $\Gamma_r(H)$ has no faces of dimension $2r$, so $\widetilde H_{2r-1}(\Gamma_r(H);\K)$ is the space of $(2r-1)$-cycles, and it suffices to exhibit a nonzero cycle.

Identify the augmented chain complex of the full simplex on $[n]$ with the exterior algebra $\mathcal E=\bigwedge_{\mathbb Z}(e_1,\dots,e_n)$ via $e_{i_1}\wedge\cdots\wedge e_{i_k}\leftrightarrow[i_1<\dots<i_k]$. The simplicial boundary is the unique $\mathbb Z$-linear map $\partial$ with $\partial e_i=1$ and $\partial(ab)=\partial(a)\,b+(-1)^{|a|}a\,\partial(b)$.
Label the vertices of $C$ cyclically as $v_1,\dots,v_\ell$ and put $c=\sum_{k\in\mathbb Z_\ell}e_{v_k}\wedge e_{v_{k+1}}$. Then $c$ has even degree and $\partial c=\sum_k(e_{v_{k+1}}-e_{v_k})=0$, so $\partial(c^r)=r\,c^{r-1}\partial c=0$.

In the expansion of $c^r$, a product of $r$ edge-blocks vanishes unless the edges are pairwise disjoint. For an $r$-matching $M$ of $C$, all $r!$ orderings of its edges give the same wedge, since the blocks are even and commute. Hence $c^r=r!\,z$ with $z=\sum_M\varepsilon_Me_{V(M)}$, $\varepsilon_M\in\{\pm1\}$, summed over the $r$-matchings $M$ of $C$ (that is, matchings consisting of edges of the cycle $C$; chords of $C$ in $H$ play no role). Since $\mathcal E$ is torsion-free and $r!\,\partial z=0$, we get $\partial z=0$ over $\mathbb Z$, hence over every field.

Since $\ell\ge2r+1$, each $V(M)$ is a proper subset of $V(C)$, so $C[V(M)]$ (the subgraph of the cycle $C$, not of $H$, induced on $V(M)$) is a disjoint union of paths and has at most one perfect matching, namely $M$. Thus distinct $M$ have distinct vertex sets, and $z$ has coefficients $\pm1$ on distinct basis elements. An $r$-matching of $C$ exists because $\lfloor\ell/2\rfloor\ge r$, so $z\ne0$ over every field. Each $M$ is an $r$-matching of $H$, so $V(M)\in\Gamma_r(H)$ by Lemma~\ref{lem:faces}. Thus $z$ is a nonzero cycle of $\Gamma_r(H)$.
\end{proof}

\subsection{A formula for the depth and the regularity of whisker graphs}\label{subsec:general-depth}

Let $H$ be a graph on $[n]$ and $G=W(H)$; no assumption on the girth of $H$ is made. For disjoint sets $Z,W\subseteq[n]$ we write $H_{Z,W}$ for the graph
on the vertex set $Z\cup W$ with edge set $E(H[Z])$; thus $H_{Z,W}$ is $H[Z]$ together with the
isolated vertices of $W$, and $H_{[n],\emptyset}=H$.

The following lemma is standard (see, e.g., \cite{J08}); we include the short argument.

\begin{lemma}\label{lem:susp}
Let $\Delta$ be a simplicial complex and $v$ a vertex such that $\Delta\setminus v$ is a cone.
If $\{v\}\notin\Delta$, then $\Delta$ is contractible. If $\{v\}\in\Delta$, then $\Delta$ is
homotopy equivalent to the suspension of $\link_\Delta(v)$, and
$\widetilde H_i(\Delta;\K)\cong\widetilde H_{i-1}(\link_\Delta(v);\K)$ for all $i$.
\end{lemma}
\begin{proof}
If $\{v\}\notin\Delta$, then $\Delta=\Delta\setminus v$ is a cone. Otherwise $\Delta$ is the
union of the contractible subcomplexes $\Delta\setminus v$ and $v*\link_\Delta(v)$, whose
intersection is $\link_\Delta(v)$. If $\link_\Delta(v)=\{\emptyset\}$, then $v$ is an isolated
vertex of $\Delta$, and $\Delta$ is homotopy equivalent to two points. Otherwise the union of two
contractible subcomplexes is homotopy equivalent to the suspension of their intersection. The
statement on homology is the Mayer--Vietoris sequence.
\end{proof}

\begin{lemma}\label{lem:link-general}
Let $\Delta=\cf^q(G)$, let $\sigma=x_U\cup y_T$ be a face of $\Delta$, and put $A=U\cap T$,
$Z=U\setminus T$, $W=T\setminus U$ and $N=[n]\setminus(U\cup T)$.
\begin{enumerate}
\item If $N=\emptyset$, then $\link_\Delta(\sigma)\cong\Theta(H_{Z,W},\,q-1-|A|)$, where
$y_i$ corresponds to $i$ for $i\in Z$ and $x_j$ corresponds to $j$ for $j\in W$.
\item If $\sigma\cup x_N\in\Delta$, then $\link_\Delta(\sigma)$ is homotopy equivalent to the
$|N|$-fold suspension of $\link_\Delta(\sigma\cup x_N)$; in particular
$\widetilde H_i(\link_\Delta(\sigma);\K)\cong\widetilde H_{i-|N|}(\link_\Delta(\sigma\cup x_N);\K)$
for all $i$. If $\sigma\cup x_N\notin\Delta$, then $\link_\Delta(\sigma)$ is contractible.
\end{enumerate}
\end{lemma}
\begin{proof}
(1) The vertices of $G$ outside $\sigma$ are $y_i$, $i\in Z$, and $x_j$, $j\in W$. Let
$F=y_{T'}\cup x_Q$ with $T'\subseteq Z$ and $Q\subseteq W$. By Lemma~\ref{lem:facets},
$\nu(G[\sigma\cup F])=|A|+|T'|+|Q|+\nu(H[Z\setminus T'])$. On the other hand,
$c_{H_{Z,W}}(T'\cup Q)=|T'|+|Q|+\nu(H[Z\setminus T'])$, because the vertices of $W$ are isolated
in $H_{Z,W}$. Hence $F\in\link_\Delta(\sigma)$ if and only if
$T'\cup Q\in\Theta(H_{Z,W},q-1-|A|)$.

(2) We use induction on $|N|$; the case $N=\emptyset$ is trivial. Let $j\in N$ and
$L=\link_\Delta(\sigma)$. If $F\in L$ and $x_j\notin F$, then $y_j$ is an isolated vertex of
$G[\sigma\cup F\cup\{y_j\}]$, so $F\cup\{y_j\}\in L$; thus $L\setminus x_j$ is a cone with apex
$y_j$. If $\sigma\cup\{x_j\}\notin\Delta$, then $L$ is contractible by Lemma~\ref{lem:susp}, and
$\sigma\cup x_N\notin\Delta$. Otherwise $\link_L(x_j)=\link_\Delta(\sigma\cup\{x_j\})$, and
Lemma~\ref{lem:susp} gives
$\widetilde H_i(L;\K)\cong\widetilde H_{i-1}(\link_\Delta(\sigma\cup\{x_j\});\K)$ and a homotopy
equivalence between $L$ and the suspension of $\link_\Delta(\sigma\cup\{x_j\})$. The claim now
follows from the induction hypothesis applied to $\sigma\cup\{x_j\}$.
\end{proof}

Recall that, by Hochster's formula for the local
cohomology of a Stanley--Reisner ring \cite[Theorem~5.3.8]{BH93}, for every simplicial complex
$\Delta$,
\begin{equation}\label{eq:hochster-local}
\begin{aligned}
\depth\K[\Delta]&=\min\{\,|\sigma|+j+1:\ \sigma\in\Delta,\ \widetilde H_j(\link_\Delta(\sigma);\K)\ne0\,\},\\
\reg\K[\Delta]&=\max\{\,j+1:\ \sigma\in\Delta,\ \widetilde H_j(\link_\Delta(\sigma);\K)\ne0\,\}.
\end{aligned}
\end{equation}
Indeed, the $i$th local cohomology module of $\K[\Delta]$ is nonzero in a multidegree with
support $\sigma$ if and only if $\sigma\in\Delta$ and
$\widetilde H_{i-|\sigma|-1}(\link_\Delta(\sigma);\K)\ne0$, and the largest total degree of such a
multidegree is $-|\sigma|$. The first formula follows since $\depth M=\min\{i:H^i_{\mathfrak m}(M)\ne0\}$,
and the second since $\reg M=\max_i\{a_i(M)+i\}$, where $a_i(M)$ is the largest degree of a nonzero
component of $H^i_{\mathfrak m}(M)$.

\begin{theorem}\label{thm:general-depth-reg}
Let $H$ be a graph on $[n]$, $G=W(H)$, $1\le q\le n$, and let $\K$ be a field. Let
$\mathcal S$ be the set of all triples $(Z,W,j)$, where $Z,W\subseteq[n]$ are disjoint, $Z\cup W\ne\emptyset$, and
$\widetilde H_j(\Gamma_{n-q}(H_{Z,W});\K)\ne0$. Then
\begin{align*}
\depth R/I(G)^{[q]}&=2n-2-\max\{\,j:\ (Z,W,j)\in\mathcal S\,\},\\
\reg R/I(G)^{[q]}&=\max\{\,2|Z|+|W|-2-j:\ (Z,W,j)\in\mathcal S\,\}.
\end{align*}
\end{theorem}
\begin{proof}
Let $\Delta=\cf^q(G)$. We evaluate \eqref{eq:hochster-local}. Let $\sigma$ be a face with data
$(A,Z_0,W,N)$ as in Lemma~\ref{lem:link-general}, and suppose
$\widetilde H_j(\link_\Delta(\sigma);\K)\ne0$. By Lemma~\ref{lem:link-general}(2),
$\sigma'=\sigma\cup x_N$ is a face and $\widetilde H_{j-|N|}(\link_\Delta(\sigma');\K)\ne0$; the
face $\sigma'$ has data $(A,Z,W,\emptyset)$ with $Z=Z_0\cup N$. Note that
$|\sigma|+j+1=|\sigma'|+(j-|N|)+1$ and $j+1\le|Z|+(j-|N|)+1$. Conversely, if $\sigma'$ is a face
with data $(A,Z,W,\emptyset)$ and $\widetilde H_{j_0}(\link_\Delta(\sigma');\K)\ne0$, then
$\sigma'\setminus x_Z$ is a face whose link has nonzero homology in degree $j_0+|Z|$, again by
Lemma~\ref{lem:link-general}(2). Therefore
\begin{align*}
\depth\K[\Delta]&=\min\{\,2|A|+|Z|+|W|+j_0+1\,\},&
\reg\K[\Delta]&=\max\{\,|Z|+j_0+1\,\},
\end{align*}
where both are taken over all partitions $[n]=A\sqcup Z\sqcup W$ such that
$\sigma'=x_{A\cup Z}\cup y_{A\cup W}\in\Delta$, and all $j_0$ with
$\widetilde H_{j_0}(\link_\Delta(\sigma');\K)\ne0$.

Fix such a partition. Since $|A|\le\nu(G[\sigma'])\le q-1<n$, we have $Z\cup W\ne\emptyset$. Put $m=|Z|+|W|=n-|A|$ and $s'=q-1-|A|$. By
Lemma~\ref{lem:link-general}(1), $\link_\Delta(\sigma')\cong\Theta(H_{Z,W},s')$. Since
$0\le s'\le m-1$, Lemma~\ref{lem:Gamma-dual} gives
\[
\Theta(H_{Z,W},s')=\Gamma_{m-s'-1}(H_{Z,W})^\vee=\Gamma_{n-q}(H_{Z,W})^\vee,
\]
and
$\Gamma_{n-q}(H_{Z,W})\ne2^{Z\cup W}$ because its dual is not void. By Theorem~\ref{thm:alex},
$\widetilde H_{j_0}(\link_\Delta(\sigma');\K)\ne0$ if and only if
$\widetilde H_j(\Gamma_{n-q}(H_{Z,W});\K)\ne0$ for $j=m-j_0-3$. With this substitution,
\[
2|A|+|Z|+|W|+j_0+1=2n-2-j,~ |Z|+j_0+1=2|Z|+|W|-2-j .
\]
Finally, let $Z,W\subseteq[n]$ be disjoint with $Z\cup W\ne\emptyset$, $A=[n]\setminus(Z\cup W)$, and suppose that
$\sigma'=x_{A\cup Z}\cup y_{A\cup W}$ is not a face, i.e., $|A|+\nu(H[Z])\ge q$. Then
$\gamma_{H_{Z,W}}(Z\cup W)=|Z|+|W|-\nu(H[Z])\le n-q$, so $\Gamma_{n-q}(H_{Z,W})$ is the full simplex on the nonempty set $Z\cup W$ and has
no reduced homology. Hence the triples in $\mathcal S$ are exactly those arising from faces
$\sigma'$ as above. Since $\K[\Delta]\ne0$ has finite depth, at least one such face $\sigma'$ with nonzero link homology exists; so $\mathcal S\ne\emptyset$ and both maxima are well defined.
\end{proof}

\subsection{Depth in terms of the girth and the circumference}\label{subsec:depth-girth}

\begin{theorem}\label{depth-formula}
Let $H$ be a graph on $[n]$, $G=W(H)$, $1\le q\le n$, $g=\girth(H)$, and $R=\K[x_1,\dots,x_n,y_1,\dots,y_n]$. Then, for every field $\K$:
\begin{enumerate}
\item $2q-1\le\depth R/I(G)^{[q]}\le n+q-1$.
\item If $q=n$ or $2(n-q)<\circum(H)$, then $\depth R/I(G)^{[q]}=2q-1$.
\item If $q\le\lfloor g/2\rfloor$, then $\depth R/I(G)^{[q]}=n+q-1$. In particular, if $H$ is a forest this holds for every $q\le n$ \textup{(}for forests this is \cite[Theorem~3.5]{DRS24}\textup{)}.
\item If $g$ is odd and $q=(g+1)/2$, then $\depth R/I(G)^{[q]}=n$.
\end{enumerate}
\end{theorem}
\begin{proof}
(1) We have $\depth R/I\le\dim R/I$ and $\dim R/I(G)^{[q]}=n+q-1$ by Lemma~\ref{lem:vwc-facets}, which gives the upper bound. The ideal $I(G)^{[q]}$ is nonzero (as $q\le n=\nu(G)$), proper, and generated in degree $2q\ge2$, so $\depth R/I(G)^{[q]}\ge2q-1$ by Lemma~\ref{lem:gendeg}.

(2) Put $r=n-q$.

\emph{Case $q=n$.} By Remark~\ref{rem:top-power}, $\depth R/I(G)^{[n]}=2n-1=2q-1$.

\emph{Case $2r<\circum(H)$ (and $q<n$).} Then $r\ge1$ and $H$ has a cycle of length
$\circum(H)\ge2r+1$, so $\widetilde H_{2r-1}(\Gamma_r(H);\K)\ne0$ by Lemma~\ref{lem:cycle}. Since
$H_{[n],\emptyset}=H$, the triple $([n],\emptyset,2r-1)$ belongs to the set $\mathcal S$ of
Theorem~\ref{thm:general-depth-reg}, and that theorem gives
$\depth R/I(G)^{[q]}\le2n-2-(2r-1)=2q-1$. Together with (1), $\depth R/I(G)^{[q]}=2q-1$.

(3) Let $q\le\lfloor g/2\rfloor$ (any $q\le n$ if $H$ is a forest). By Corollary~\ref{main-cor}(2), $I(G)^{[q]}$ is Cohen--Macaulay over every field, so $\depth R/I(G)^{[q]}=\dim R/I(G)^{[q]}=n+q-1$.

(4) Let $g$ be odd and $q=(g+1)/2$; note $q\le n$ since $g\le n$. Let $d'$ be the minimum cardinality of a facet of $\cf^q(G)$.

\emph{Claim 1: $d'\ge n$.} By Lemma~\ref{lem:facets}, every facet contains a vertex of each pair, so $d'\ge n$.

\emph{Claim 2: $d'\le n$.} Let $C$ be a cycle of $H$ of length $g$. It has no chord (a chord would create a shorter cycle), so $H[V(C)]=C$. Put $W=[n]\setminus V(C)$, so $|W|=n-g$, and
$F=x_{V(C)}\cup y_W .$
Then $G[F]$ is a copy of $C$ together with isolated vertices, so $\nu(G[F])=(g-1)/2=q-1$ and $F$ is a face with $|F|=n$. It is a facet, since the remaining vertices are $x_s$ with $s\in W$ and $y_i$ with $i\in V(C)$, and:
\begin{itemize}
\item for $s\in W$, the whisker $x_sy_s$ together with a $(q-1)$-matching of $C$ is a $q$-matching of $G[F\cup\{x_s\}]$;
\item for $i\in V(C)$, the edge $x_iy_i$ together with a perfect matching of the path $C\setminus x_i$ (which has $g-1$ vertices) is a $q$-matching of $G[F\cup\{y_i\}]$.
\end{itemize}
Hence $d'=n$.
Since $q=\lceil g/2\rceil$, Theorem~\ref{reg:main} shows that $\cf^q(G)$ is shellable, hence sequentially Cohen--Macaulay over every field (Corollary~\ref{main-cor}(1)). By Lemma \ref{lem:seqCM-depth}, $\depth R/I(G)^{[q]}=d'=n$.
\end{proof}

For whisker cycles, Theorem~\ref{depth-formula} determines the depth completely.

\begin{corollary}\label{cor:cycle-depth}
Let $G=W(C_n)$ with $n\ge3$. Then, for every field $\K$,
\[
\depth R/I(G)^{[q]}=
\begin{cases}
n+q-1, & 1\le q\le\lfloor n/2\rfloor,\\[2pt]
2q-1, & \lfloor n/2\rfloor<q\le n.
\end{cases}
\]
In particular, Conjecture~6.3 of~\cite{DRS24} holds for all $n\ge3$ and $1\le q\le n$, and
$R/I(G)^{[q]}$ is Cohen--Macaulay if and only if $q\le\lfloor n/2\rfloor$ or $q=n$.
\end{corollary}
\begin{proof}
If $q\le\lfloor n/2\rfloor=\lfloor\girth(C_n)/2\rfloor$, Theorem~\ref{depth-formula}(3) gives $n+q-1$. If $q>\lfloor n/2\rfloor$, then $q>n/2$, so $2(n-q)<n=\circum(C_n)$, and Theorem~\ref{depth-formula}(2) gives $2q-1$.
The last assertion follows because $\dim R/I(G)^{[q]}=n+q-1$, and $2q-1=n+q-1$ only for $q=n$.
\end{proof}

\begin{remark}\label{rem:monotone}
By Corollary~\ref{cor:cycle-depth}, the depth function $q\mapsto\depth R/I(W(C_n))^{[q]}$ is not
monotone for $n\ge5$; for $n=5$ it is $5,6,5,7,9$. On the other hand, the \emph{normalized depth
function} $g_I(q)=\depth R/I^{[q]}-(2q-1)$ of~\cite{EHHS24} equals $n-q$ for $q\le\lfloor n/2\rfloor$
and $0$ for $q>\lfloor n/2\rfloor$, so it is non-increasing. Erey, Herzog, Hibi and Saeedi~Madani
conjectured that the normalized depth function of every squarefree monomial ideal is
non-increasing~\cite{EHHS24}; this fails in general~\cite{SF24inc} but remains open for edge
ideals, where it is known for forests~\cite[Theorem~2.1]{CFE25} and further classes~\cite{FHH23}.
\end{remark}

Theorem~\ref{depth-formula} recovers \cite[Theorem~3.5, Corollary~3.6 and Theorem~5.8]{DRS24}
and the case of whisker graphs of \cite[Corollary~2.10]{FM24}.

\subsection{The Cohen--Macaulay range of whisker graphs}\label{subsec:cm-range}

We now determine exactly for which $q$ the ring $R/I(W(H))^{[q]}$ is Cohen--Macaulay. The
missing ingredient is the following generalization of Example~\ref{sharp-bd}.

\begin{proposition}\label{prop:evencycle}
Let $H$ be a graph on $[n]$, $G=W(H)$, and let $C$ be an induced cycle of $H$ of even length
$2m\ge4$. Let $V_0$ and $V_1$ be the two color classes of $C$. If $m+1\le q\le n-m+1$, then
$\cf^q(G)$ has a face $\sigma$ with $|\sigma|=n+q-m-1$ such that
\[
\link_{\cf^q(G)}(\sigma)\cong 2^{V_0}\sqcup 2^{V_1}.
\]
Consequently, $\cf^q(G)$ is not sequentially Cohen--Macaulay, and
$\depth R/I(G)^{[q]}\le n+q-m$ over every field.
\end{proposition}
\begin{proof}
Since $0\le q-1-m\le n-2m$, we can choose $A\subseteq[n]\setminus V(C)$ with $|A|=q-1-m$. Put
$Z=V(C)$, $W=[n]\setminus(A\cup Z)$ and $\sigma=x_{A\cup Z}\cup y_{A\cup W}$. By
Lemma~\ref{lem:facets} and since $C$ is induced, $\nu(G[\sigma])=|A|+\nu(C)=q-1$, so
$\sigma\in\cf^q(G)$, and $|\sigma|=2|A|+|Z|+|W|=n+q-m-1$. The face $\sigma$ meets every pair,
so by Lemma~\ref{lem:link-general}(1)
\[
\link_{\cf^q(G)}(\sigma)\cong\Theta(K,m),~ K=H_{Z,W},
\]
where $K$ is the disjoint union of $C$ and the isolated vertices of $W$. For $T\subseteq V(K)$
write $T_C=T\cap V(C)$ and $T_W=T\cap W$; then $c_K(T)=c_C(T_C)+|T_W|$. By \eqref{eq:c},
$c_C(T_C)\ge c_C(\emptyset)=\nu(C)=m$, so $T\in\Theta(K,m)$ forces $T_W=\emptyset$ and
$c_C(T_C)=m$. As in Example~\ref{sharp-bd}, $c_C(T_C)=m$ if $T_C\subseteq V_0$ or
$T_C\subseteq V_1$, while $c_C(\{i,j\})=m+1$ for $i\in V_0$, $j\in V_1$, so by \eqref{eq:c} no face
meets both $V_0$ and $V_1$. Hence $\Theta(K,m)=2^{V_0}\sqcup2^{V_1}$, the disjoint union of two simplices of
dimension $m-1\ge1$.

This link has $1$-dimensional faces in two connected components, so it is not sequentially
Cohen--Macaulay by Lemma~\ref{lem:scm-conn}, and neither is $\cf^q(G)$ by
Lemma~\ref{lem:scm-links}. Moreover $\widetilde H_0$ of the link is nonzero, and
\eqref{eq:hochster-local} gives $\depth R/I(G)^{[q]}\le|\sigma|+1=n+q-m$.
\end{proof}

\begin{theorem}\label{thm:cm-range}
Let $H$ be a graph on $[n]$ which is not a forest, let $g=\girth(H)$, $G=W(H)$ and
$1\le q\le n$. Then, over every field,
\[
R/I(G)^{[q]}\ \text{is Cohen--Macaulay}\iff q\le\lfloor g/2\rfloor\ \text{or}\ q=n .
\]
If $H$ is a forest, then $R/I(G)^{[q]}$ is Cohen--Macaulay for every $1\le q\le n$.
\end{theorem}
\begin{proof}
If $H$ is a forest, or if $q\le\lfloor g/2\rfloor$, the ring is Cohen--Macaulay by
Corollary~\ref{main-cor}(2). For $q=n$ it is Cohen--Macaulay because its depth $2n-1$
\textup{(}Theorem~\ref{depth-formula}(2)\textup{)} equals its dimension. Now let
$\lfloor g/2\rfloor<q\le n-1$. Since $\dim R/I(G)^{[q]}=n+q-1$, it suffices to show that the
depth is smaller. A shortest cycle of $H$ has no chord, so it is an induced cycle of length
$g\le\circum(H)\le n$.

\emph{$g=2\mu+1$ odd.} The shortest odd cycle of $H$ has length $g$. If $\mu+1\le q\le n-\mu$,
then $\cf^q(G)$ is not pure by Corollary~\ref{pure-range}, hence not Cohen--Macaulay. If
$n-\mu<q\le n-1$, then $2(n-q)<2\mu<g\le\circum(H)$, and
$\depth R/I(G)^{[q]}=2q-1<n+q-1$ by Theorem~\ref{depth-formula}(2).

\emph{$g=2m$ even.} Then $m\ge2$. If $m+1\le q\le n-m$, then
$\depth R/I(G)^{[q]}\le n+q-m<n+q-1$ by Proposition~\ref{prop:evencycle}, applied to a shortest
cycle. If $n-m<q\le n-1$, then $2(n-q)<2m=g\le\circum(H)$, and
$\depth R/I(G)^{[q]}=2q-1<n+q-1$ by Theorem~\ref{depth-formula}(2).
\end{proof}

\begin{remark}\label{rem:cm-range}
Proposition~\ref{prop:evencycle} also shows that, when the girth $g=2m$ of $H$ is even,
$\cf^q(G)$ is not sequentially Cohen--Macaulay for $m+1\le q\le n-m+1$; together with
Theorem~\ref{reg:main} this determines the sequentially Cohen--Macaulay range of $\cf^q(W(H))$
except for $n-m+1<q<n$. This gap is closed in Corollary~\ref{cor:scm-trianglefree}, which determines
the sequentially Cohen--Macaulay range for every triangle-free graph $H$.
\end{remark}

Combining Theorem~\ref{thm:cm-range} with Section~\ref{sec:rigid}, we obtain the complete
answer for very well-covered graphs.

\begin{corollary}\label{cor:vwc-complete}
Let $G$ be a very well-covered graph on $2n$ vertices and $1\le q\le n$. Then, over every
field, $R/I(G)^{[q]}$ is Cohen--Macaulay if and only if one of the following holds:
\begin{enumerate}
\item $G=W(H)$ for a forest $H$;
\item $G=W(H)$ for a graph $H$ with a cycle, and either $q\le\lfloor\girth(H)/2\rfloor$ or $q=n$;
\item $G$ is Cohen--Macaulay but not a whisker graph, and $q\in\{1,n\}$;
\item $G$ is not Cohen--Macaulay, and $q=n$.
\end{enumerate}
\end{corollary}
\begin{proof}
Cases (1) and (2) are Theorem~\ref{thm:cm-range}. A Cohen--Macaulay very well-covered graph
which is not a whisker graph has $n\ge3$ by Lemma~\ref{lem:config}, so case (3) is
Corollary~\ref{cor:vwc-range}. Case (4) is Theorem~\ref{thm:noncm-vwc}.
\end{proof}

The two extreme cases $q=2$ and $q=n-1$ are worth recording.

\begin{corollary}\label{cm-chra}
Let $H$ be a graph on $n$ vertices and $G=W(H)$.
\begin{enumerate}
\item If $n\ge2$, then $\cf^{2}(G)$ is Cohen--Macaulay if and only if $H$ contains no triangle.
\item If $n\ge3$, then $\cf^{n-1}(G)$ is Cohen--Macaulay if and only if $H$ is a forest.
\item If $G'$ is any very well-covered graph on $2n\ge6$ vertices, then $\cf^2(G')$ is
Cohen--Macaulay if and only if $G'=W(H')$ for a graph $H'$ without triangles, and
$\cf^{n-1}(G')$ is Cohen--Macaulay if and only if $G'=W(H')$ for a forest $H'$.
\end{enumerate}
\end{corollary}
\begin{proof}
If $H$ is a forest, both complexes are Cohen--Macaulay by Theorem~\ref{thm:cm-range}. Let $H$ have a
cycle, so $n\ge3$ and $g=\girth(H)\le n$. By Theorem~\ref{thm:cm-range}, $\cf^2(G)$ is
Cohen--Macaulay if and only if $2\le\lfloor g/2\rfloor$, that is, $g\ge4$; and $\cf^{n-1}(G)$ is
Cohen--Macaulay if and only if $n-1\le\lfloor g/2\rfloor\le n/2$, which is impossible for $n\ge3$.
For (3), $2\le q\le n-1$ for $q\in\{2,n-1\}$, so by Corollary~\ref{cor:vwc-complete} the complex
can only be Cohen--Macaulay if $G'$ is a whisker graph, and (1), (2) apply.
\end{proof}

Corollaries~\ref{cor:vwc-complete} and~\ref{cm-chra} sharpen \cite[Corollary~2.8]{FM24}, which
states that a very well-covered graph $G$ has all squarefree powers $I(G)^{[q]}$,
$1\le q\le\nu(G)$, Cohen--Macaulay if and only if $G$ is a Cohen--Macaulay forest.

\subsection{Very well-covered graphs that are not whisker graphs}\label{subsec:nonwhisker}

By Corollary~\ref{cor:vwc-range}, the squarefree powers $I(G)^{[q]}$, $2\le q\le n-1$, of a
Cohen--Macaulay very well-covered graph $G$ that is not a whisker graph are not
Cohen--Macaulay. We now show that their depth takes the smallest possible value $2q-1$ for all
sufficiently large $q$. Let $G$ be a very well-covered graph on $2n$ vertices, labeled as in
Theorem~\ref{thm:vwc}, which contains one of the configurations \textup{(A)}, \textup{(B)} of
Lemma~\ref{lem:config} on a set $\Sigma$ of three indices. As in the proof of
Theorem~\ref{thm:rigid} we put $a=x_i$ and $b=y_k$ in case \textup{(A)}, and $a=x_i$ and
$b=x_m$ in case \textup{(B)}; in both cases $\{a,b\}\in E(G)$ and
$\{x_l,y_l\}\subseteq N[a]\cup N[b]$ for all $l\in\Sigma$. We denote by $p$ the number of
indices $l\in[n]\setminus\Sigma$ with $\{x_l,y_l\}\not\subseteq N[a]\cup N[b]$; thus
$0\le p\le n-3$. The number $p$ depends on the chosen configuration; the following theorem holds for every choice.

\begin{theorem}\label{thm:nonwhisker-depth}
Let $G$ be a very well-covered graph on $2n$ vertices containing a configuration \textup{(A)}
or \textup{(B)}, and let $p$ be as above. Then, for all $2\le q\le n-1$,
\[
2q-1\ \le\ \depth R/I(G)^{[q]}\ \le\ \max\{2q-1,\ q+p+1\}.
\]
In particular, $\depth R/I(G)^{[q]}=2q-1$ for all $p+2\le q\le n$.
\end{theorem}
\begin{proof}
The lower bound is Lemma~\ref{lem:gendeg}. For $q=n=\nu(G)$ we have $\depth R/I(G)^{[n]}=2n-1$ by
Remark~\ref{rem:top-power}. Let
$2\le q\le n-1$ and $t=\min\{q-2,p\}$. Since $q-2\le n-3$, we can choose $q-2$ indices outside
$\Sigma$ in such a way that $t$ of them are among the $p$ indices $l$ with
$\{x_l,y_l\}\not\subseteq N[a]\cup N[b]$. Let $\rho$ be the union of the corresponding pairs and
$G'=G\setminus\rho$. By Corollary~\ref{cor:pair-link},
$\link_{\cf^q(G)}(\rho)=\cf^2(G')$. The graph $G'$ satisfies \textup{(V1)}--\textup{(V3)} and
contains the configuration on $\Sigma$, so the hypotheses of Lemma~\ref{lem:two-edges} hold for
$G'$ and the edge $\{a,b\}$, as shown in Step~1 of the proof of Theorem~\ref{thm:rigid}. Hence
there is a face $\sigma=\{a,b\}\cup\tau$ of $\cf^2(G')$ with disconnected link, where $\tau$ is
an independent set of $G'$ consisting of vertices outside $N[a]\cup N[b]$. Such vertices belong
to the $p-t$ pairs of $G'$ that are not contained in $N[a]\cup N[b]$, and an independent set
contains at most one vertex of each pair; so $|\tau|\le p-t$. Now
$\link_{\cf^q(G)}(\rho\cup\sigma)=\link_{\cf^2(G')}(\sigma)$ is disconnected, that is,
$\widetilde H_0(\link_{\cf^q(G)}(\rho\cup\sigma);\K)\ne0$, and \eqref{eq:hochster-local} gives
\[
\depth R/I(G)^{[q]}\le|\rho\cup\sigma|+1=2(q-2)+2+|\tau|+1\le2q-1+p-t=\max\{2q-1,\ q+p+1\}.
\qedhere
\]
\end{proof}

\begin{corollary}\label{cor:depth-n-1}
Let $G$ be a Cohen--Macaulay very well-covered graph on $2n$ vertices, $n\ge3$. Then
\[
\depth R/I(G)^{[n-1]}=
\begin{cases}
2n-2, & \text{if $G$ is a forest},\\
2n-3, & \text{otherwise}.
\end{cases}
\]
\end{corollary}
\begin{proof}
A Cohen--Macaulay forest is a whisker graph \cite[Theorem~2.4]{vill_cohen}, so if $G$ is not a whisker graph, then $G$ is not a forest. In this case $G$ contains a configuration \textup{(A)} or \textup{(B)} by
Lemma~\ref{lem:config}, and $q=n-1\ge p+2$; so the depth is $2q-1=2n-3$ by
Theorem~\ref{thm:nonwhisker-depth}. Let $G=W(H)$; then $G$ is a forest if and only if $H$ is.
If $H$ is a forest, the depth is $n+q-1=2n-2$ by Theorem~\ref{depth-formula}(3). If $H$ has a
cycle, then $2(n-q)=2<\circum(H)$, and the depth is $2q-1=2n-3$ by
Theorem~\ref{depth-formula}(2).
\end{proof}

\section{The sequentially Cohen--Macaulay range of whisker graphs}\label{sec:scm}

In this section we use the cost complexes $\Theta(K,s)$ of Definition~\ref{def:cost-complex} and the
complexes $\Gamma_r(K)$ of Definition~\ref{def:Gamma} for arbitrary graphs $K$. If $W$ is a finite set
disjoint from $V(K)$, we write $K\sqcup W$ for the graph obtained from $K$ by adding the elements of $W$
as isolated vertices. A \emph{hole} of a graph $H$ is an induced cycle of length at least $4$.

Francisco and Van Tuyl~\cite[Proposition~4.1]{FVT07} showed that the cycle $C_n$ is sequentially Cohen--Macaulay only for
$n\in\{3,5\}$. We show that for the squarefree powers of a whisker graph every hole of $H$ is an
obstruction, while triangles are not. The key input is the Alexander duality between sequential
Cohen--Macaulayness and componentwise linearity: by~\cite[Theorem~2.1]{HH99} and~\cite{HRW99}, a
simplicial complex $\Delta$ is sequentially Cohen--Macaulay over $\K$ if and only if the
Stanley--Reisner ideal $I_{\Delta^\vee}$ is componentwise linear over $\K$.

\begin{lemma}\label{lem:isolated-link}
Let $K$ be a graph, $W$ a finite set disjoint from $V(K)$, and $s\ge0$. If $U\subseteq W$ and
$|U|+\nu(K)\le s$, then $U\in\Theta(K\sqcup W,s)$ and
$\link_{\Theta(K\sqcup W,s)}(U)=\Theta\bigl(K\sqcup(W\setminus U),\,s-|U|\bigr).$
\end{lemma}

\begin{proof}
Since the vertices of $W$ are isolated, $c_{K\sqcup W}(T)=c_K(T\cap V(K))+|T\cap W|$ for every
$T\subseteq V(K)\cup W$. In particular $c_{K\sqcup W}(U)=\nu(K)+|U|\le s$, and for
$T\subseteq(V(K)\cup W)\setminus U$ we have $c_{K\sqcup W}(T\cup U)=c_{K\sqcup(W\setminus U)}(T)+|U|$.
\end{proof}

\begin{lemma}\label{lem:cycle-cost}
Let $\ell\ge3$ and $C=C_\ell$. Then $c_C(\emptyset)=\lfloor\ell/2\rfloor$, and
$c_C(T)\ge\lceil\ell/2\rceil$ for every nonempty $T\subseteq V(C)$.
\end{lemma}

\begin{proof}
The first assertion is $\nu(C_\ell)=\lfloor\ell/2\rfloor$. Let $|T|=t\ge1$. As in the proof of
Lemma~\ref{lem:mindeg2}, $C\setminus T$ is a disjoint union of at most $t$ paths with $\ell-t$
vertices in total, so $\nu(C\setminus T)\ge(\ell-2t)/2$. Hence $c_C(T)\ge\ell/2$, and
$c_C(T)\ge\lceil\ell/2\rceil$ since $c_C(T)$ is an integer.
\end{proof}

\begin{lemma}\label{lem:theta-cycle}
Let $\ell\ge4$ and $\lceil\ell/2\rceil\le s\le\ell-2$. Then $\Theta(C_\ell,s)$ is not sequentially
Cohen--Macaulay over any field.
\end{lemma}

\begin{proof}
Put $C=C_\ell$ and $r=\ell-s-1$, so that $1\le r\le\lfloor\ell/2\rfloor-1$. By
Lemma~\ref{lem:Gamma-dual}, $\Theta(C,s)=\Gamma_r(C)^\vee$, hence $\Theta(C,s)^\vee=\Gamma_r(C)$.
By the result quoted above it suffices to show that $I=I_{\Gamma_r(C)}$ is not componentwise linear.

We first show that $I$ is generated in degree $r+1$. Let $\alpha$ denote the independence number.
If $D\subsetneq V(C)$, then $C[D]$ is a disjoint union of paths. A path $P_k$ on $k$ vertices satisfies
$\gamma_{P_k}(V(P_k))=k-\lfloor k/2\rfloor=\lceil k/2\rceil=\alpha(P_k)$, and both $\gamma$ and $\alpha$ are
additive over connected components, so $\gamma_C(D)=\alpha(C[D])$. Thus a proper subset $D$ is a nonface
of $\Gamma_r(C)$ if and only if it contains an independent set of size $r+1$. The set $V(C)$ is a
nonface, since $\gamma_C(V(C))=\lceil\ell/2\rceil\ge r+1$, and it contains an independent set of size
$\lfloor\ell/2\rfloor\ge r+1$. Every independent set of size $r+1$ is a nonface, and all its proper
subsets are faces because $\gamma_C(D)\le|D|$. Hence the minimal nonfaces of $\Gamma_r(C)$ are exactly the
independent sets of size $r+1$, and $I$ is generated in degree $r+1$.

A componentwise linear ideal generated in a single degree has a linear resolution, so it suffices to
show that $I$ does not. Since $\ell\ge2r+2$, Lemma~\ref{lem:cycle} gives
$\widetilde H_{2r-1}(\Gamma_r(C);\K)\ne0$. By Theorem~\ref{HT-result}(1) applied to $\sigma=V(C)$,
we get $\beta_{\ell-2r,\ell}(R/I)\ne0$, where $\ell-2r\ge2$. If $I$ had a linear resolution, every
nonzero $\beta_{i,j}(R/I)$ with $i\ge1$ would satisfy $j=i+r$. Here $j-i=2r\ne r$, a contradiction.
\end{proof}

\begin{lemma}\label{lem:theta-cycle-iso}
Let $\ell\ge4$, let $W$ be a set of $w\ge0$ isolated vertices, and let
$\lceil\ell/2\rceil\le s\le\ell+w-2$. Then $\Theta(C_\ell\sqcup W,s)$ is not sequentially
Cohen--Macaulay over any field.
\end{lemma}

\begin{proof}
Put $C=C_\ell$ and $s_0=\lceil\ell/2\rceil$. By Lemma~\ref{lem:scm-links}, a complex having a face
whose link is not sequentially Cohen--Macaulay is itself not sequentially Cohen--Macaulay. We use the
links of Lemma~\ref{lem:isolated-link}.

\emph{Case $s-w\ge s_0$.} Take $U=W$. Then $|U|+\nu(C)=w+\lfloor\ell/2\rfloor\le s$, and the link
of $U$ is $\Theta(C,s-w)$ with $s_0\le s-w\le\ell-2$. This complex is not sequentially
Cohen--Macaulay by Lemma~\ref{lem:theta-cycle}.

\emph{Case $s-w<s_0$.} Then $w\ge1$. Choose $U\subseteq W$ with $|U|=s-s_0$, so that $0\le|U|<w$ and
$|U|+\nu(C)\le s$. The link of $U$ is $\Delta=\Theta(C\sqcup W',s_0)$ with
$W'=W\setminus U\ne\emptyset$. We determine $\Delta$. Let $T=T_C\cup T_{W'}$ with
$T_C\subseteq V(C)$ and $T_{W'}\subseteq W'$.
\begin{itemize}
\item If $T_C\ne\emptyset$ and $T_{W'}\ne\emptyset$, then by Lemma~\ref{lem:cycle-cost} the cost of $T$
is at least $s_0+1$, so $T\notin\Delta$.
\item If $T_C=\emptyset$, the cost of $T$ is $\lfloor\ell/2\rfloor+|T_{W'}|$. So $T\in\Delta$ if and
only if $T_{W'}=\emptyset$ when $\ell$ is even, and if and only if $|T_{W'}|\le1$ when $\ell$ is odd.
\end{itemize}
Hence $\Delta=\Theta(C,s_0)$ if $\ell$ is even, and $\Delta$ is the disjoint union of
$\Theta(C,s_0)$ with the $|W'|$ isolated points of $W'$ if $\ell$ is odd.

The complex $\Theta(C,s_0)$ has dimension at least $1$: the complement $T$ of an independent set of
size $\lfloor\ell/2\rfloor$ satisfies $c_C(T)=|T|=s_0$, and $s_0\ge2$. Adding isolated points to a
complex of dimension at least $1$ does not change its pure $i$-skeleta for $i\ge1$, and pure
$0$-skeleta are always Cohen--Macaulay. By Duval's criterion~\cite[Theorem~3.3]{Duval}, $\Delta$ is
sequentially Cohen--Macaulay if and only if $\Theta(C,s_0)$ is. The latter is not, by
Lemma~\ref{lem:theta-cycle}, since $s_0\le\ell-2$ for $\ell\ge4$.
\end{proof}

\begin{theorem}\label{thm:holes}
Let $H$ be a graph on $[n]$ with a hole $C$ of length $\ell$. Then, over every field, $\cf^q(W(H))$
is not sequentially Cohen--Macaulay for any $q$ with $\lceil\ell/2\rceil<q<n$.
\end{theorem}

\begin{proof}
Put $Z=V(C)$, $W=[n]\setminus Z$ and $\sigma=x_Z\cup y_W$. In the notation of
Lemma~\ref{lem:link-general} we have $A=N=\emptyset$. The graph $G[\sigma]$ is $H[Z]=C$ together
with the isolated vertices $y_j$, $j\in W$, so $\nu(G[\sigma])=\lfloor\ell/2\rfloor\le q-1$, and
$\sigma$ is a face of $\cf^q(G)$. By Lemma~\ref{lem:link-general}(1),
\[
\link_{\cf^q(G)}(\sigma)\cong\Theta(H_{Z,W},q-1)=\Theta(C\sqcup W,q-1),
\]
since $C$ is an induced subgraph and the vertices of $W$ are isolated in $H_{Z,W}$. Here
$|W|=n-\ell$, and the hypothesis on $q$ gives $\lceil\ell/2\rceil\le q-1\le n-2=\ell+|W|-2$. By
Lemma~\ref{lem:theta-cycle-iso} the link is not sequentially Cohen--Macaulay, and hence neither is
$\cf^q(G)$, by Lemma~\ref{lem:scm-links}.
\end{proof}

\begin{corollary}\label{cor:scm-trianglefree}
Let $H$ be a triangle-free graph on $[n]$ which is not a forest, let $g=\girth(H)$, and let
$1\le q\le n$. Then, over every field, $R/I(W(H))^{[q]}$ is sequentially Cohen--Macaulay if and only if
\[
q\le\lceil g/2\rceil~\text{or}~ q=n.
\]
\end{corollary}

\begin{proof}
A shortest cycle of $H$ has no chord, so it is a hole of length $g\ge4$. If $\lceil g/2\rceil<q<n$,
the ring is not sequentially Cohen--Macaulay by Theorem~\ref{thm:holes}. If $q\le\lceil g/2\rceil$, it
is sequentially Cohen--Macaulay by Corollary~\ref{main-cor}(1). For $q=n$ it is Cohen--Macaulay by
Theorem~\ref{thm:cm-range}.
\end{proof}

\begin{corollary}\label{cor:chordal}
If $R/I(W(H))^{[q]}$ is sequentially Cohen--Macaulay for every $1\le q\le n$, then $H$ is chordal.
\end{corollary}

\begin{proof}
Suppose that $H$ has a hole of length $\ell$. Since $\ell\ge4$, we have
$\lceil\ell/2\rceil+1\le\ell-1\le n-1$, and Theorem~\ref{thm:holes} applied with
$q=\lceil\ell/2\rceil+1$ gives a contradiction.
\end{proof}

For a graph $H$ let $h(H)$ be the length of a shortest hole of $H$, with $h(H)=\infty$ if $H$ is
chordal. By Theorem~\ref{thm:holes}, $\cf^q(W(H))$ is not sequentially Cohen--Macaulay for
$\lceil h(H)/2\rceil<q<n$. We expect that this is the only obstruction.

\begin{question}\label{conj:hole}
Let $H$ be a graph on $[n]$ and $1\le q\le n$. Then $\cf^q(W(H))$ is sequentially Cohen--Macaulay if
and only if $q\le\lceil h(H)/2\rceil$ or $q=n$. In particular, all squarefree powers of $I(W(H))$ are
sequentially Cohen--Macaulay if and only if $H$ is chordal.
\end{question}

By Corollary~\ref{cor:scm-trianglefree}, Question~\ref{conj:hole} has an affirmative answer for triangle-free graphs. In general,
the missing ingredient is an analogue of Theorem~\ref{reg:main} in which the girth is replaced by
$h(H)$. A natural approach is to replace the leaf in Case~2 of the proof of
Proposition~\ref{prop:Theta-VD} by a simplicial vertex.

\section{Homology, Betti numbers and regularity}\label{sec:topology}

In this section we study the homology of matching-free complexes and its consequences for the
graded Betti numbers and the regularity of squarefree powers. We first compute the homology in
the top degree for every graph with a perfect matching (Theorem~\ref{thm:pm-top}); this determines
when the regularity of $R/I(G)^{[q]}$ attains its largest possible value. For whisker graphs we
then describe the whole homology of $\cf^q(G)$ (Theorem~\ref{thm:general-homology}) and its
homotopy type in the Cohen--Macaulay range (Theorem~\ref{thm:homotopy}).

\subsection{Top homology for graphs with a perfect matching}\label{subsec:pm}

In this subsection $G$ is an arbitrary graph on $2n$ vertices with a perfect matching
$\{e_1,\dots,e_n\}$, and $R=\K[x_v:v\in V(G)]$. For $F\subseteq V(G)$ we put
$A(F)=\{i\in[n]:e_i\subseteq F\}$, $B(F)=\{i\in[n]:|e_i\cap F|=1\}.$
We determine the homology of $\cf^q(G)$ in the largest possible degree $n+q-2$; no assumption
on $G$ other than the existence of a perfect matching is needed. In particular, the results
apply to all very well-covered graphs.

\begin{lemma}\label{lem:pm-top}
Let $1\le q\le n$. Every face of $\cf^q(G)$ has at most $n+q-1$ elements. If $F$ is a face with
$n+q-1$ elements, then $|A(F)|=q-1$, $B(F)=[n]\setminus A(F)$, and
$F\cap\bigcup_{i\in B(F)}e_i$ is an independent set of $G$.
\end{lemma}
\begin{proof}
Let $F$ be a face, $a=|A(F)|$ and $b=|B(F)|$. The edges $e_i$, $i\in A(F)$, form a matching of
$G[F]$, so $a\le q-1$, and $|F|=2a+b\le n+a\le n+q-1$ because $a+b\le n$. If equality holds,
then $a=q-1$ and $a+b=n$. Moreover, an edge of $G$ with both ends in
$F\cap\bigcup_{i\in B(F)}e_i$, together with the edges $e_i$, $i\in A(F)$, would be a
$q$-matching of $G[F]$.
\end{proof}

\begin{lemma}\label{lem:blocks}
Let $V=V_1\sqcup\dots\sqcup V_k$ be a partition of a finite set into nonempty blocks, let
$0\le m\le k-1$, and let $\mathcal B$ be the simplicial complex of all $F\subseteq V$ that contain
at most $m$ of the blocks. Then $\widetilde H_i(\mathcal B;\K)\cong\K^{\binom{k-1}{m}}$ for
$i=|V|-k+m-1$, and $\widetilde H_i(\mathcal B;\K)=0$ otherwise.
\end{lemma}
\begin{proof}
Put $S=\K[x_v:v\in V]$ and $u_j=\prod_{v\in V_j}x_v$. The minimal nonfaces of $\mathcal B$ are the
unions of $m+1$ blocks, so $I_{\mathcal B}$ is generated by the products of $m+1$ distinct monomials
among $u_1,\dots,u_k$. Let $Q=\K[t_1,\dots,t_k]$, and let $J\subset Q$ be the ideal generated by
all squarefree monomials of degree $m+1$; thus $I_{\mathcal B}=JS$ under the homomorphism
$Q\to S$, $t_j\mapsto u_j$. Since the monomials $u_j$ involve pairwise disjoint sets of
variables, $S$ is a free $Q$-module. Hence a minimal multigraded free resolution of $Q/J$ over
$Q$ becomes, after tensoring with $S$, a minimal multigraded free resolution of $S/I_{\mathcal B}$,
in which the multidegree $t_1\cdots t_k$ is replaced by $\prod_{v\in V}x_v$, and no other
multidegree is. Therefore
$\beta_{i,V}(S/I_{\mathcal B})=\beta_{i,[k]}(Q/J)$ for all $i$. The Stanley--Reisner complex of $J$ is
the $(m-1)$-skeleton of the simplex on $[k]$, whose reduced homology is $\K^{\binom{k-1}{m}}$ in
degree $m-1$ and zero in all other degrees. By Theorem~\ref{HT-result}(1), applied to $Q/J$ and
to $S/I_{\mathcal B}$,
\[
\dim_\K\widetilde H_{|V|-i-1}({\mathcal B};\K)=\beta_{i,V}(S/I_{\mathcal B})=\beta_{i,[k]}(Q/J)
=\dim_\K\widetilde H_{k-i-1}\bigl(\text{skeleton};\K\bigr),
\]
which is $\binom{k-1}{m}$ for $i=k-m$ and zero otherwise.
\end{proof}

\begin{theorem}\label{thm:pm-top}
Let $G$ be a graph on $2n$ vertices with a perfect matching and with $k$ connected components,
and let $1\le q\le n$. Then, for every field $\K$,
\[
\dim_\K\widetilde H_{n+q-2}\bigl(\cf^q(G);\K\bigr)=\binom{k-1}{n-q},
\]
and $\widetilde H_i(\cf^q(G);\K)=0$ for all $i>n+q-2$.
\end{theorem}
\begin{proof}
Let $\Delta=\cf^q(G)$ and $d=n+q-2$. By Lemma~\ref{lem:pm-top}, $\dim\Delta\le d$. Hence
$\widetilde H_i(\Delta;\K)=0$ for $i>d$, and $\widetilde H_d(\Delta;\K)$ is the space of
$d$-cycles $z=\sum_Fc_FF$, the sum being over the (oriented) $d$-faces of $\Delta$. Let $\Pi$ be
the graph on $[n]$ in which $i\ne j$ are adjacent if and only if some edge of $G$ joins a vertex
of $e_i$ to a vertex of $e_j$; the connected components of $\Pi$ correspond to those of $G$.
Fix a $d$-cycle $z$.

\emph{Step 1.} Let $F$ be a $d$-face with $c_F\ne0$, let $l\in B(F)$, and write
$e_l=\{s,s'\}$ with $s\in F$. By Lemma~\ref{lem:pm-top}, every $d$-face meets $e_l$, so the only
$d$-faces that can contain $F\setminus\{s\}$ are $F$ and $F_l=(F\setminus\{s\})\cup\{s'\}$. The
coefficient of $F\setminus\{s\}$ in $\partial z$ is $\pm c_F\pm c_{F_l}$, where the second
term is omitted if $F_l\notin\Delta$. Hence $F_l$ is a $d$-face and $c_{F_l}\ne0$. Note that
$A(F_l)=A(F)$.

\emph{Step 2.} We claim that if $c_F\ne0$, then no two elements of $B(F)$ lie in the same
connected component of $\Pi$. We show by induction on $t\ge1$ that, for every $d$-face $F$ with
$c_F\ne0$, no two elements of $B(F)$ are at distance $t$ in $\Pi$. Let $t=1$, and suppose that $l,l'\in B(F)$ are adjacent in $\Pi$;
thus there is an edge $\{u,w\}$ of $G$ with $u\in e_l$ and $w\in e_{l'}$. Applying Step~1 at
most twice, we obtain a $d$-face $F'$ with $A(F')=A(F)$ and $u,w\in F'$, which contradicts
Lemma~\ref{lem:pm-top}. Now let $t\ge2$, and suppose that $l,l'\in B(F)$ are at distance $t$.
Let $i$ be a neighbor of $l$ in $\Pi$ at distance $t-1$ from $l'$. By the case $t=1$,
$i\notin B(F)$, so $i\in A(F)$. Let $v\in e_i$. The vertices outside $F\setminus\{v\}$ are $v$
and, for each $b\in B(F)$, the vertex $s_b'$ of $e_b$ not in $F$. Hence a $d$-face containing
$F\setminus\{v\}$ is $F$ or $F^b=(F\setminus\{v\})\cup\{s_b'\}$ for some $b\in B(F)$, and
$B(F^b)=(B(F)\setminus\{b\})\cup\{i\}$. If $b\ne l'$, then $B(F^b)$ contains $i$ and $l'$, which
are at distance $t-1$; if $b=l'$, it contains $i$ and $l$, which are at distance $1$. In both
cases $c_{F^b}=0$ by the induction hypothesis (applied to the distances $t-1$ and $1$). So the coefficient of $F\setminus\{v\}$ in
$\partial z$ is $\pm c_F\ne0$, a contradiction.

\emph{Step 3.} Let $V_1,\dots,V_k$ be the vertex sets of the connected components of $G$, put
$m=k-(n-q+1)$, and let $\mathcal B$ be the complex of all $F\subseteq V(G)$ that contain at most $m$
of the sets $V_j$ (the void complex if $m<0$). Each $V_j$ is a union of edges $e_i$. If
$F\in\mathcal B$, then $F$ misses a vertex of at least $n-q+1$ components, and for each of these
$\nu(G[F\cap V_j])\le|V_j|/2-1$; hence $\nu(G[F])\le n-(n-q+1)=q-1$. Thus $\mathcal B\subseteq\Delta$.
Conversely, if $F$ is a $d$-face with $c_F\ne0$, then by Step~2 the $n-q+1$ edges $e_l$,
$l\in B(F)$, lie in pairwise different components of $G$, none of which is contained in $F$; so
$F\in\mathcal B$. Consequently the $d$-cycles of $\Delta$ are exactly the $d$-cycles of $\mathcal B$, and
since $\dim\mathcal B\le d$,
$\widetilde H_d(\Delta;\K)=\widetilde H_d(\mathcal B;\K)$. If $m<0$, this space is zero, and
$\binom{k-1}{n-q}=0$ because $n-q>k-1$. If $m\ge0$, then $m\le k-1$, and Lemma~\ref{lem:blocks}
with $|V|=2n$ gives $\widetilde H_i(\mathcal B;\K)\cong\K^{\binom{k-1}{m}}$ for
$i=2n-k+m-1=n+q-2$, where $\binom{k-1}{m}=\binom{k-1}{n-q}$.
\end{proof}

\begin{corollary}\label{cor:pm-reg}
Let $G$ be a graph on $2n$ vertices with a perfect matching and with $k$ connected components,
and let $1\le q\le n$. Then:
\begin{enumerate}
\item $\beta_{n-q+1,2n}(R/I(G)^{[q]})=\binom{k-1}{n-q}$, and $\beta_{i,2n}(R/I(G)^{[q]})=0$ for
$i<n-q+1$.
\item $\reg R/I(G)^{[q]}=n+q-1$ if and only if $q\ge n-k+1$. In this case
$\beta_{n-q+1,2n}(R/I(G)^{[q]})$ is an extremal Betti number of $R/I(G)^{[q]}$, and it is the
unique extremal Betti number if $R/I(G)^{[q]}$ is Cohen--Macaulay.
\end{enumerate}
\end{corollary}
\begin{proof}
Let $\Delta=\cf^q(G)$.
(1) By Theorem~\ref{HT-result}(1) for the full vertex set,
$\beta_{i,2n}(R/I(G)^{[q]})=\dim_\K\widetilde H_{2n-i-1}(\Delta;\K)$, and the claim follows from
Theorem~\ref{thm:pm-top}.

(2) By Theorem~\ref{HT-result}(1), $\reg R/I(G)^{[q]}$ is the largest $j+1$ such that
$\widetilde H_j(\Delta_\sigma;\K)\ne0$ for some $\sigma\subseteq V(G)$. Since
$\dim\Delta\le n+q-2$, we get $\reg R/I(G)^{[q]}\le n+q-1$; for graphs with a perfect matching this recovers the bound $\reg I(G)^{[q]}\le\nu(G)+q$, which Seyed Fakhari proved for
graphs with a matching of size $\lfloor|V(G)|/2\rfloor$, in particular for graphs with a perfect
matching \cite{Sey24,Fa25}; for arbitrary graphs this inequality is a conjecture of Erey, Herzog,
Hibi and Saeedi~Madani \cite{EHHM22}.
A cycle of dimension $n+q-2$ of an
induced subcomplex $\Delta_\sigma$ is a cycle of $\Delta$, and it is not a boundary in either
complex; so equality holds if and only if $\widetilde H_{n+q-2}(\Delta;\K)\ne0$, that is, by
Theorem~\ref{thm:pm-top}, if and only if $n-q\le k-1$. In this case let
$\beta_{i,i+j}(R/I(G)^{[q]})\ne0$ with $i\ge n-q+1$ and $j\ge n+q-1$. Then $j=n+q-1$ because
$j\le\reg R/I(G)^{[q]}$, and $i+j\le2n$ because $I(G)^{[q]}$ is squarefree; hence $i=n-q+1$. If $R/I(G)^{[q]}$ is
Cohen--Macaulay, then, since $\widetilde H_{n+q-2}(\Delta;\K)\ne0$, the complex $\Delta$ has a face
with $n+q-1$ elements, so $\depth R/I(G)^{[q]}=\dim R/I(G)^{[q]}=n+q-1$ and
$\pd R/I(G)^{[q]}=n-q+1$ by the Auslander--Buchsbaum formula, so every
nonzero $\beta_{i,i+j}$ has $i\le n-q+1$ and $j\le n+q-1$, and the extremal Betti number is unique.
\end{proof}

\begin{remark}\label{rem:pm-reg}
\textup{(1)} If $G$ is very well-covered, then $\dim R/I(G)^{[q]}=n+q-1$ by
Lemma~\ref{lem:vwc-facets}, so Corollary~\ref{cor:pm-reg}(2) characterizes the equality
$\reg R/I(G)^{[q]}=\dim R/I(G)^{[q]}$.

\textup{(2)} For a whisker graph $G=W(H)$, the number $k$ is the number of connected components of
$H$; see Theorem~\ref{thm:homotopy} for the consequences in the Cohen--Macaulay range.
\end{remark}

\begin{remark}\label{rem:aim-question}
Chau, Das, Roy and Saha~\cite[Theorem~3.16]{CDRS25} proved that
$\reg I(G)^{[q]}\ge\mathrm{aim}(G,q)+q$ for every graph $G$ and every $1\le q\le\nu(G)$, where
$\mathrm{aim}(G,q)$ is the largest size of a \emph{$q$-admissible} matching, that is, a matching $M$
with a partition $M=M_1\sqcup\dots\sqcup M_r$ such that every edge of $G[V(M)]$ lies in some
$G[V(M_i)]$ and $|M|\le r+q-1$ \cite[Definition~3.12]{CDRS25}. In \cite[Question~6.1]{CDRS25} it is
asked whether equality holds for all graphs $G$ with $\reg I(G)=\nu_1(G)+1$, where $\nu_1(G)$ is the
induced matching number; this class contains the very well-covered graphs. Let $G$ have a perfect
matching, $2n=|V(G)|$, and $k$ connected components, and let $q\ge n-k+1$. The perfect matching of
$G$, partitioned according to the connected components of $G$, is $q$-admissible, so
$\mathrm{aim}(G,q)=n$, and the lower bound above, together with Seyed Fakhari's upper bound $\reg I(G)^{[q]}\le n+q$~\cite{Sey24,Fa25}, already gives $\reg I(G)^{[q]}=n+q$; this is the ``if'' part of Corollary~\ref{cor:pm-reg}(2). Thus equality
holds for every graph with a perfect matching, in particular for every very well-covered graph, in
the range $q\ge n-k+1$, and this follows from known results. The new information in Corollary~\ref{cor:pm-reg} is the ``only if'' part, namely $\reg I(G)^{[q]}<n+q$ for $q\le n-k$, and the value of the Betti number. For $q\le n-k$ the question of~\cite[Question~6.1]{CDRS25} remains open for these graphs.
\end{remark}

\subsection{Homology of matching-free complexes of whisker graphs}\label{subsec:whisker-hom}

Let $H$ be a graph on $[n]$, $G=W(H)$, and $1\le q\le n$. The next theorem describes the homology
of $\cf^q(G)$ for all $H$ and all $q$ in terms of the complexes $\Gamma_r$ of
Section~\ref{subsec:gamma}.

\begin{theorem}\label{thm:general-homology}
Let $H$ be a graph on $[n]$, $G=W(H)$ and $1\le q\le n$.
\begin{enumerate}
\item If $q\le\nu(H)$, then $\cf^q(G)$ is contractible.
\item If $q>\nu(H)$, then $\cf^q(G)$ is homotopy equivalent to the $n$-fold suspension of
$\Theta(H,q-1)=\Gamma_{n-q}(H)^\vee$, and for every field $\K$ and every $i$,
\[
\dim_\K\widetilde H_i\bigl(\cf^q(G);\K\bigr)=\dim_\K\widetilde H_{2n-3-i}\bigl(\Gamma_{n-q}(H);\K\bigr).
\]
\end{enumerate}
\end{theorem}
\begin{proof}
Apply Lemma~\ref{lem:link-general}(2) to $\sigma=\emptyset$, for which $N=[n]$; its proof shows
that $\cf^q(G)$ is contractible if $x_{[n]}\notin\cf^q(G)$, and is homotopy equivalent to the
$n$-fold suspension of $\link_{\cf^q(G)}(x_{[n]})$ otherwise. Since $\nu(G[x_{[n]}])=\nu(H)$,
this proves (1). Let $q>\nu(H)$. By Lemma~\ref{lem:link-general}(1) with $Z=[n]$ and
$A=W=\emptyset$, $\link_{\cf^q(G)}(x_{[n]})\cong\Theta(H,q-1)$, and
$\Theta(H,q-1)=\Gamma_{n-q}(H)^\vee$ by Lemma~\ref{lem:Gamma-dual}, where $\Gamma_{n-q}(H)\ne2^{[n]}$ because
$\gamma_H([n])=n-\nu(H)>n-q$.
Hence $\widetilde H_i(\cf^q(G);\K)\cong\widetilde H_{i-n}(\Gamma_{n-q}(H)^\vee;\K)
\cong\widetilde H^{\,2n-3-i}(\Gamma_{n-q}(H);\K)$ by Theorem~\ref{thm:alex}, and over a field
homology and cohomology have the same dimension.
\end{proof}

In the Cohen--Macaulay range the homotopy type is determined by Theorem~\ref{thm:pm-top}.

\begin{theorem}\label{thm:homotopy}
Let $H$ be a graph on $[n]$ with $k$ connected components, $G=W(H)$, and let
$1\le q\le\min\{n,\lfloor\girth(H)/2\rfloor\}$ or $q=n$. Then $\cf^q(G)$ is homotopy equivalent to a wedge of
$\binom{k-1}{n-q}$ spheres of dimension $n+q-2$. Moreover, the following are equivalent:
\begin{enumerate}
\item[(i)] $\cf^q(G)$ is collapsible;
\item[(ii)] $\cf^q(G)$ is contractible;
\item[(iii)] $\widetilde H_*(\cf^q(G);\K)=0$ for some \textup{(}equivalently, every\textup{)} field $\K$;
\item[(iv)] $q\le n-k$.
\end{enumerate}
These conditions hold whenever $H$ contains a cycle and $q\le\lfloor\girth(H)/2\rfloor$.
\end{theorem}
\begin{proof}
Let $\Delta=\cf^q(G)$ and $d=n+q-2$. If $q=n$, then $\Delta$ is the boundary of the simplex on $V(G)$
\textup{(}Remark~\ref{rem:top-power}\textup{)}, a sphere of dimension $2n-2=d$.
This is a wedge of $1=\binom{k-1}{0}$ sphere, it is not contractible, and $q=n>n-k$, so all
assertions hold. Now let $q\le\min\{n,\lfloor\girth(H)/2\rfloor\}$. By Theorem~\ref{reg:main} and Corollary~\ref{main-cor}, $\Delta$ is
shellable and pure of dimension $d$. A pure shellable $d$-dimensional complex is homotopy
equivalent to a wedge of $d$-spheres, one for each facet $F_j$ of a shelling with
$\operatorname{Rest}(F_j)=F_j$ (Theorem~\ref{thm:standard}(2)). The number of spheres is therefore $\dim_\K\widetilde H_d(\Delta;\K)$, which is
$\binom{k-1}{n-q}$ by Theorem~\ref{thm:pm-top}, since $G$ has $k$ connected components. This number
is zero if and only if $n-q>k-1$, which proves (ii)$\Leftrightarrow$(iii)$\Leftrightarrow$(iv). A
collapsible complex is contractible, and if no facet of a shelling satisfies
$\operatorname{Rest}(F_j)=F_j$, then reading the shelling backwards gives a sequence of elementary collapses
from $\Delta$ to a vertex \cite{Bjo95,Koz08}, \cite[Section~4]{Ziv09}; so (iv) implies (i).
If $H$ contains a cycle, then the component of $H$ containing a shortest cycle has at least
$\girth(H)$ vertices, so $n-k\ge\girth(H)-1\ge\lfloor\girth(H)/2\rfloor\ge q$.
\end{proof}

We conclude the paper with two questions. The first is Question~\ref{conj:hole}, which asks for
the sequentially Cohen--Macaulay range of $\cf^q(W(H))$ when $H$ has triangles. The second concerns the
complexes $\Gamma_r$, which by Theorem~\ref{thm:general-depth-reg} control the depth of all squarefree
powers of whisker graphs. Let $L$ be a connected graph and $r\ge1$. By Lemma~\ref{lem:faces},
$\dim\Gamma_r(L)\le2r-1$, and by Lemma~\ref{lem:cycle},
$\widetilde H_{2r-1}(\Gamma_r(L);\K)\ne0$ whenever $L$ contains a cycle of length at least $2r+1$.
We ask whether the homology of $\Gamma_r(L)$ is concentrated in this top degree.

\begin{question}\label{conj:conc}
Let $L$ be a connected graph and $r\ge1$.
\begin{enumerate}
\item[(a)]  $\widetilde H_i(\Gamma_r(L);\K)=0$ for all $i\ne2r-1$.
\item[(b)]  If $r\ge2$ and $\widetilde H_{2r-1}(\Gamma_r(L);\K)\ne0$, then
$\widetilde H_{2r-3}(\Gamma_{r-1}(L);\K)\ne0$.
\end{enumerate}
\end{question}

Both statements in Question~\ref{conj:conc} hold for trees. Indeed, let $v$ be a leaf of a graph $L$ with neighbor $u$, and let
$r\ge1$. For $D\subseteq V(L)\setminus\{u,v\}$, Lemma~\ref{lem:pend} gives
$\gamma_L(D\cup\{v\})=\gamma_L(D\cup\{u,v\})=\gamma_L(D)+1$, so $\link_{\Gamma_r(L)}(v)$ is the cone
$u*\Gamma_{r-1}(L\setminus\{u,v\})$, and $\Gamma_r(L)$ is homotopy equivalent to
$\Gamma_r(L\setminus v)$. Deleting leaves repeatedly shows that $\Gamma_r(L)$ is contractible if $L$
is a tree.

Our interest in Question~\ref{conj:conc} comes from the normalized depth function of
Remark~\ref{rem:monotone}. By Theorem~\ref{thm:general-depth-reg}, the depth of
$R/I(W(H))^{[q]}$ is governed by the homology of the complexes $\Gamma_r(H_{Z,W})$, and the
connected components of the graphs $H_{Z,W}$ are connected induced subgraphs of $H$. We expect
that statements \textup{(a)} and \textup{(b)} of Question~\ref{conj:conc} for these subgraphs
determine the homology of all complexes $\Gamma_r(H_{Z,W})$, and that they imply that the
normalized depth function of $I(W(H))$ is non-increasing, together with an explicit formula for
$\depth R/I(W(H))^{[q]}$. We will return to this elsewhere.

\vspace*{2mm}
\noindent
\textbf{Acknowledgments.}
The authors acknowledge support from the Science and Engineering Research Board (SERB),
and the second author acknowledges support from the National Board for Higher Mathematics (NBHM).

\vspace*{1mm}
\noindent
\textbf{Data availability statement.}  
Data sharing is not applicable to this article as no datasets were generated or analyzed during the current study.

\vspace*{1mm}
\noindent
\textbf{Conflict of interest.}  
The authors declare that they have no known competing financial interests or personal relationships that could have appeared to influence the work reported in this paper.

\vspace*{1mm}
\noindent\textbf{Declaration of AI-assisted writing.}
The authors used OpenAI's ChatGPT to assist in improving the language
and presentation of the manuscript. The authors have reviewed
and edited the AI-assisted content and take full responsibility for
the final version of the manuscript.

\bibliographystyle{abbrv}
\bibliography{refs} 

@article{HH99,
  author  = {Herzog, J{\"u}rgen and Hibi, Takayuki},
  title   = {Componentwise linear ideals},
  journal = {Nagoya Math. J.},
  volume  = {153},
  year    = {1999},
  pages   = {141--153}
}

@article{HRW99,
  author  = {Herzog, J{\"u}rgen and Reiner, Victor and Welker, Volkmar},
  title   = {Componentwise linear ideals and {G}olod rings},
  journal = {Michigan Math. J.},
  volume  = {46},
  number  = {2},
  year    = {1999},
  pages   = {211--223}
}

@article{BWW09, author={A. Bj\"orner and M. L. Wachs and V. Welker},
   title={On sequentially {C}ohen--{M}acaulay complexes and posets},
   journal={Israel J. Math.}, volume={169}, year={2009}, pages={295--316}}

@incollection {Bjo95,
    AUTHOR = {Bj\"orner, A.},
     TITLE = {Topological methods},
 BOOKTITLE = {Handbook of combinatorics, {V}ol.\ 1,\ 2},
     PAGES = {1819--1872},
 PUBLISHER = {Elsevier Sci. B. V., Amsterdam},
      YEAR = {1995},
      ISBN = {0-444-88002-X},
   MRCLASS = {52B05 (05B35 05E25 52B40)},
  MRNUMBER = {1373690},
MRREVIEWER = {Andrew\ Vince},
}

@article{EF26,
    AUTHOR = {Erey, Nursel and Ficarra, Antonino},
     TITLE = {Matching powers of monomial ideals and edge ideals of weighted
              oriented graphs},
   JOURNAL = {J. Algebra Appl.},
  FJOURNAL = {Journal of Algebra and its Applications},
    VOLUME = {25},
      YEAR = {2026},
    NUMBER = {11},
     PAGES = {Paper No. 2650118, 19},
      ISSN = {0219-4988,1793-6829},
   MRCLASS = {13F20 (05E40 13A70)},
  MRNUMBER = {5063036},
MRREVIEWER = {Abdelamir\ Dabbabi},
       DOI = {10.1142/S0219498826501185},
       URL = {https://doi.org/10.1142/S0219498826501185},
}

@ARTICLE{FM24,
       author = {{Ficarra}, Antonino and {Moradi}, Somayeh},
        title = "{Monomial ideals whose all matching powers are Cohen-Macaulay}",
      journal = {arXiv e-prints},
         year = 2024,
        month = oct,
          eid = {arXiv:2410.01666},
        pages = {arXiv:2410.01666},
          doi = {10.48550/arXiv.2410.01666},
archivePrefix = {arXiv},
       eprint = {2410.01666},
 primaryClass = {math.AC},
       adsurl = {https://ui.adsabs.harvard.edu/abs/2024arXiv241001666F}
}

@article {CFE25,
    AUTHOR = {Crupi, Marilena and Ficarra, Antonino and Lax, Ernesto},
     TITLE = {Matchings, squarefree powers, and {B}etti splittings},
   JOURNAL = {Illinois J. Math.},
  FJOURNAL = {Illinois Journal of Mathematics},
    VOLUME = {69},
      YEAR = {2025},
    NUMBER = {2},
     PAGES = {353--372},
      ISSN = {0019-2082,1945-6581},
   MRCLASS = {13C15 (05E40 13A70)},
  MRNUMBER = {4919940},
MRREVIEWER = {Matthew\ Weaver},
       DOI = {10.1215/00192082-11919226},
       URL = {https://doi.org/10.1215/00192082-11919226},
}

@article {FVT07,
    AUTHOR = {Francisco, Christopher A. and Van Tuyl, Adam},
     TITLE = {Sequentially {C}ohen-{M}acaulay edge ideals},
   JOURNAL = {Proc. Amer. Math. Soc.},
  FJOURNAL = {Proceedings of the American Mathematical Society},
    VOLUME = {135},
      YEAR = {2007},
    NUMBER = {8},
     PAGES = {2327--2337},
      ISSN = {0002-9939,1088-6826},
   MRCLASS = {13F55 (05E25)},
  MRNUMBER = {2302553},
MRREVIEWER = {Carles\ Bivi\`a-Ausina},
       DOI = {10.1090/S0002-9939-07-08841-7},
       URL = {https://doi.org/10.1090/S0002-9939-07-08841-7},
}

@article {EH25,
    AUTHOR = {Erey, Nursel and Hibi, Takayuki},
     TITLE = {Density of linearity index in the interval of matching
              numbers},
   JOURNAL = {J. Algebraic Combin.},
  FJOURNAL = {Journal of Algebraic Combinatorics. An International Journal},
    VOLUME = {61},
      YEAR = {2025},
    NUMBER = {4},
     PAGES = {Paper No. 52, 7},
      ISSN = {0925-9899,1572-9192},
   MRCLASS = {05E40 (05C70 13D02)},
  MRNUMBER = {4921584},
MRREVIEWER = {Jorge\ Neves},
       DOI = {10.1007/s10801-025-01421-7},
       URL = {https://doi.org/10.1007/s10801-025-01421-7},
}

@article {Fa25,
    AUTHOR = {Seyed Fakhari, S. A.},
     TITLE = {Matchings and {C}astelnuovo-{M}umford regularity of squarefree
              powers of edge ideals},
   JOURNAL = {J. Commut. Algebra},
  FJOURNAL = {Journal of Commutative Algebra},
    VOLUME = {17},
      YEAR = {2025},
    NUMBER = {2},
     PAGES = {203--207},
      ISSN = {1939-0807,1939-2346},
   MRCLASS = {13D02 (05E40 13A70)},
  MRNUMBER = {4954445},
       DOI = {10.1216/jca.2025.17.203},
       URL = {https://doi.org/10.1216/jca.2025.17.203},
}

@ARTICLE{DRS25,    
    AUTHOR = {Das, Kanoy and Roy, Amit and Saha, Kamalesh},
     TITLE = {Square-free powers of {C}ohen-{M}acaulay simplicial forests},
   JOURNAL = {Proc. Amer. Math. Soc.},
  FJOURNAL = {Proceedings of the American Mathematical Society},
    VOLUME = {154},
      YEAR = {2026},
    NUMBER = {2},
     PAGES = {549--565},
      ISSN = {0002-9939,1088-6826},
   MRCLASS = {13H10 (05E40 13C15 13F55)},
  MRNUMBER = {5016541},
       DOI = {10.1090/proc/17470},
       URL = {https://doi.org/10.1090/proc/17470},
}

@article {EHHS24,
    AUTHOR = {Erey, Nursel and Herzog, J\"urgen and Hibi, Takayuki and
              Saeedi Madani, Sara},
     TITLE = {The normalized depth function of squarefree powers},
   JOURNAL = {Collect. Math.},
  FJOURNAL = {Collectanea Mathematica},
    VOLUME = {75},
      YEAR = {2024},
    NUMBER = {2},
     PAGES = {409--423},
      ISSN = {0010-0757,2038-4815},
   MRCLASS = {13C15 (05C70 05E40)},
  MRNUMBER = {4724120},
MRREVIEWER = {Yi-Huang\ Shen},
       DOI = {10.1007/s13348-023-00392-x},
       URL = {https://doi.org/10.1007/s13348-023-00392-x},
}

@article {EH21,
    AUTHOR = {Erey, Nursel and Hibi, Takayuki},
     TITLE = {Squarefree powers of edge ideals of forests},
   JOURNAL = {Electron. J. Combin.},
  FJOURNAL = {Electronic Journal of Combinatorics},
    VOLUME = {28},
      YEAR = {2021},
    NUMBER = {2},
     PAGES = {Paper No. 2.32, 16},
      ISSN = {1077-8926},
   MRCLASS = {05E40 (05C05 13D02)},
  MRNUMBER = {4272721},
MRREVIEWER = {Mehrdad\ Nasernejad},
       DOI = {10.37236/10038},
       URL = {https://doi.org/10.37236/10038},
}

@article {CF26,
    AUTHOR = {Crupi, Marilena and Ficarra, Antonino},
     TITLE = {Edge ideals whose all matching powers are
              bi-{C}ohen-{M}acaulay},
   JOURNAL = {Comm. Algebra},
  FJOURNAL = {Communications in Algebra},
    VOLUME = {54},
      YEAR = {2026},
    NUMBER = {2},
     PAGES = {661--668},
      ISSN = {0092-7872,1532-4125},
   MRCLASS = {13C14 (05E40 13C05 13C15)},
  MRNUMBER = {4996679},
       DOI = {10.1080/00927872.2025.2537272},
       URL = {https://doi.org/10.1080/00927872.2025.2537272},
}

@article {HS25,
    AUTHOR = {Hibi, Takayuki and Seyed Fakhari, Seyed Amin},
     TITLE = {Bounded powers of edge ideals: regularity and linear
              quotients},
   JOURNAL = {Proc. Amer. Math. Soc.},
  FJOURNAL = {Proceedings of the American Mathematical Society},
    VOLUME = {153},
      YEAR = {2025},
    NUMBER = {11},
     PAGES = {4619--4631},
      ISSN = {0002-9939,1088-6826},
   MRCLASS = {13D02 (05E40)},
  MRNUMBER = {4971559},
       DOI = {10.1090/proc/17376},
       URL = {https://doi.org/10.1090/proc/17376},
}

@article {TT12,
    AUTHOR = {Terai, Naoki and Trung, Ngo Viet},
     TITLE = {Cohen-{M}acaulayness of large powers of {S}tanley-{R}eisner
              ideals},
   JOURNAL = {Adv. Math.},
  FJOURNAL = {Advances in Mathematics},
    VOLUME = {229},
      YEAR = {2012},
    NUMBER = {2},
     PAGES = {711--730},
      ISSN = {0001-8708,1090-2082},
   MRCLASS = {13F55 (13H10)},
  MRNUMBER = {2855076},
MRREVIEWER = {Amir\ Mafi},
       DOI = {10.1016/j.aim.2011.10.004},
       URL = {https://doi.org/10.1016/j.aim.2011.10.004},
}

@article {CN76,
    AUTHOR = {Cowsik, R. C. and Nori, M. V.},
     TITLE = {On the fibres of blowing up},
   JOURNAL = {J. Indian Math. Soc. (N.S.)},
  FJOURNAL = {The Journal of the Indian Mathematical Society. New Series},
    VOLUME = {40},
      YEAR = {1976},
    NUMBER = {1-4},
     PAGES = {217--222},
      ISSN = {0019-5839,2455-6475},
   MRCLASS = {14M10},
  MRNUMBER = {572990},
}

@book {BH93,
    AUTHOR = {Bruns, Winfried and Herzog, J\"urgen},
     TITLE = {Cohen-{M}acaulay rings},
    SERIES = {Cambridge Studies in Advanced Mathematics},
    VOLUME = {39},
 PUBLISHER = {Cambridge University Press, Cambridge},
      YEAR = {1993},
     PAGES = {xii+403},
      ISBN = {0-521-41068-1},
   MRCLASS = {13H10 (13-02)},
  MRNUMBER = {1251956},
MRREVIEWER = {Matthew\ Miller},
}

@article {BW97,
    AUTHOR = {Bj\"orner, Anders and Wachs, Michelle L.},
     TITLE = {Shellable nonpure complexes and posets. {II}},
   JOURNAL = {Trans. Amer. Math. Soc.},
  FJOURNAL = {Transactions of the American Mathematical Society},
    VOLUME = {349},
      YEAR = {1997},
    NUMBER = {10},
     PAGES = {3945--3975},
      ISSN = {0002-9947,1088-6850},
   MRCLASS = {06A08 (05E99)},
  MRNUMBER = {1401765},
MRREVIEWER = {Volkmar\ Welker},
       DOI = {10.1090/S0002-9947-97-01838-2},
       URL = {https://doi.org/10.1090/S0002-9947-97-01838-2},
}

@article {BW96,
    AUTHOR = {Bj\"orner, Anders and Wachs, Michelle L.},
     TITLE = {Shellable nonpure complexes and posets. {I}},
   JOURNAL = {Trans. Amer. Math. Soc.},
  FJOURNAL = {Transactions of the American Mathematical Society},
    VOLUME = {348},
      YEAR = {1996},
    NUMBER = {4},
     PAGES = {1299--1327},
      ISSN = {0002-9947,1088-6850},
   MRCLASS = {06A08 (05E99 52B99)},
  MRNUMBER = {1333388},
MRREVIEWER = {T.\ S.\ Blyth},
       DOI = {10.1090/S0002-9947-96-01534-6},
       URL = {https://doi.org/10.1090/S0002-9947-96-01534-6},
}

@article {EHHM22,
    AUTHOR = {Erey, Nursel and Herzog, J\"urgen and Hibi, Takayuki and
              Saeedi Madani, Sara},
     TITLE = {Matchings and squarefree powers of edge ideals},
   JOURNAL = {J. Combin. Theory Ser. A},
  FJOURNAL = {Journal of Combinatorial Theory. Series A},
    VOLUME = {188},
      YEAR = {2022},
     PAGES = {Paper No. 105585, 24},
      ISSN = {0097-3165,1096-0899},
   MRCLASS = {05E40 (13A70)},
  MRNUMBER = {4364935},
MRREVIEWER = {Mehrdad\ Nasernejad},
       DOI = {10.1016/j.jcta.2021.105585},
       URL = {https://doi.org/10.1016/j.jcta.2021.105585},
}

@book {Koz08,
    AUTHOR = {Kozlov, Dmitry},
     TITLE = {Combinatorial algebraic topology},
    SERIES = {Algorithms and Computation in Mathematics},
    VOLUME = {21},
 PUBLISHER = {Springer, Berlin},
      YEAR = {2008},
     PAGES = {xx+389},
      ISBN = {978-3-540-71961-8},
   MRCLASS = {55-02 (05C15 05C25 06A07 52B70 55U10 57-02)},
  MRNUMBER = {2361455},
MRREVIEWER = {Rade\ \v Zivaljevi\'c},
       DOI = {10.1007/978-3-540-71962-5},
       URL = {https://doi.org/10.1007/978-3-540-71962-5},
}

@ARTICLE{DGS26,
       author = {{Das}, Sanjoy and {Ghosh}, Arka and {Selvaraja}, S},
        title = "{Regularity of Squarefree Powers of Edge Ideals of Whiskered Cycles}",
      journal = {arXiv e-prints},
         year = 2026,
        month = apr,
          eid = {arXiv:2604.17100},
        pages = {arXiv:2604.17100},
          doi = {10.48550/arXiv.2604.17100},
archivePrefix = {arXiv},
       eprint = {2604.17100},
 primaryClass = {math.AC},
       adsurl = {https://ui.adsabs.harvard.edu/abs/2026arXiv260417100D}
}

@ARTICLE{Faridi13,
       author = {{Faridi}, Sara},
        title = "{The projective dimension of sequentially Cohen-Macaulay monomial ideals}",
      journal = {arXiv e-prints},
         year = 2013,
        month = oct,
          eid = {arXiv:1310.5598},
        pages = {arXiv:1310.5598},
          doi = {10.48550/arXiv.1310.5598},
archivePrefix = {arXiv},
       eprint = {1310.5598},
 primaryClass = {math.AC},
       adsurl = {https://ui.adsabs.harvard.edu/abs/2013arXiv1310.5598F}
}

@article {FHH23,
    AUTHOR = {Ficarra, Antonino and Herzog, J\"urgen and Hibi, Takayuki},
     TITLE = {Behaviour of the normalized depth function},
   JOURNAL = {Electron. J. Combin.},
  FJOURNAL = {Electronic Journal of Combinatorics},
    VOLUME = {30},
      YEAR = {2023},
    NUMBER = {2},
     PAGES = {Paper No. 2.31, 16},
      ISSN = {1077-8926},
   MRCLASS = {13C15 (05C70 05E40)},
  MRNUMBER = {4596339},
MRREVIEWER = {Nursel\ Erey},
       DOI = {10.37236/11611},
       URL = {https://doi.org/10.37236/11611},
}

@book {LP09,
    AUTHOR = {Lov\'asz, L\'aszl\'o{} and Plummer, Michael D.},
     TITLE = {Matching theory},
      NOTE = {Corrected reprint of the 1986 original [MR0859549]},
 PUBLISHER = {AMS Chelsea Publishing, Providence, RI},
      YEAR = {2009},
     PAGES = {xxxiv+554},
      ISBN = {978-0-8218-4759-6},
   MRCLASS = {90C10 (05B35 05C70 90C05)},
  MRNUMBER = {2536865},
       DOI = {10.1090/chel/367},
       URL = {https://doi.org/10.1090/chel/367},
}

@article {BCP99,
    AUTHOR = {Bayer, Dave and Charalambous, Hara and Popescu, Sorin},
     TITLE = {Extremal {B}etti numbers and applications to monomial ideals},
   JOURNAL = {J. Algebra},
  FJOURNAL = {Journal of Algebra},
    VOLUME = {221},
      YEAR = {1999},
    NUMBER = {2},
     PAGES = {497--512},
      ISSN = {0021-8693,1090-266X},
   MRCLASS = {13D02},
  MRNUMBER = {1726711},
MRREVIEWER = {David\ A.\ Jorgensen},
       DOI = {10.1006/jabr.1999.7970},
       URL = {https://doi.org/10.1006/jabr.1999.7970},
}

@incollection {H77,
    AUTHOR = {Hochster, Melvin},
     TITLE = {Cohen-{M}acaulay rings, combinatorics, and simplicial
              complexes},
 BOOKTITLE = {Ring theory, {II} ({P}roc. {S}econd {C}onf., {U}niv.
              {O}klahoma, {N}orman, {O}kla., 1975)},
    SERIES = {Lect. Notes Pure Appl. Math.},
    VOLUME = {Vol. 26},
     PAGES = {171--223},
 PUBLISHER = {Dekker, New York-Basel},
      YEAR = {1977},
   MRCLASS = {14M05 (13H10)},
  MRNUMBER = {441987},
MRREVIEWER = {Gerhard\ Pfister},
}

@ARTICLE{DRS24,
       author = {{Das}, Kanoy Kumar and {Roy}, Amit and {Saha}, Kamalesh},
        title = "{Square-free powers of Cohen-Macaulay forests, cycles, and whiskered cycles}",
      journal = {arXiv e-prints},
         year = 2024,
        month = sep,
          eid = {arXiv:2409.06021},
        pages = {arXiv:2409.06021},
          doi = {10.48550/arXiv.2409.06021},
archivePrefix = {arXiv},
       eprint = {2409.06021},
 primaryClass = {math.AC},
       adsurl = {https://ui.adsabs.harvard.edu/abs/2024arXiv240906021D}
}

@article {Sey24,
    AUTHOR = {Seyed Fakhari, S. A.},
     TITLE = {On the {C}astelnuovo-{M}umford regularity of squarefree powers
              of edge ideals},
   JOURNAL = {J. Pure Appl. Algebra},
  FJOURNAL = {Journal of Pure and Applied Algebra},
    VOLUME = {228},
      YEAR = {2024},
    NUMBER = {3},
     PAGES = {Paper No. 107488, 12},
      ISSN = {0022-4049,1873-1376},
   MRCLASS = {13D02 (05C70 05E40)},
  MRNUMBER = {4624455},
MRREVIEWER = {Jorge\ Neves},
       DOI = {10.1016/j.jpaa.2023.107488},
       URL = {https://doi.org/10.1016/j.jpaa.2023.107488},
}

@article {DochEng09,
    AUTHOR = {Dochtermann, Anton and Engstr\"om, Alexander},
     TITLE = {Algebraic properties of edge ideals via combinatorial
              topology},
   JOURNAL = {Electron. J. Combin.},
  FJOURNAL = {Electronic Journal of Combinatorics},
    VOLUME = {16},
      YEAR = {2009},
    NUMBER = {2, Special volume in honor of Anders Bj\"orner},
     PAGES = {Research Paper 2, 24},
      ISSN = {1077-8926},
   MRCLASS = {13F55 (05C10 13D02)},
  MRNUMBER = {2515765},
MRREVIEWER = {Christopher A. Francisco},
       URL = {http://www.combinatorics.org/Volume_16/Abstracts/v16i2r2.html},
}

@book {Stanbook,
    AUTHOR = {Stanley, Richard P.},
     TITLE = {Combinatorics and commutative algebra},
    SERIES = {Progress in Mathematics},
    VOLUME = {41},
   EDITION = {Second},
 PUBLISHER = {Birkh\"auser Boston, Inc., Boston, MA},
      YEAR = {1996},
     PAGES = {x+164},
      ISBN = {0-8176-3836-9},
   MRCLASS = {05-02 (13F50)},
  MRNUMBER = {1453579},
MRREVIEWER = {Volkmar Welker},
}

@article {vill_cohen,
    AUTHOR = {Villarreal, Rafael H.},
     TITLE = {Cohen-{M}acaulay graphs},
   JOURNAL = {Manuscripta Math.},
  FJOURNAL = {Manuscripta Mathematica},
    VOLUME = {66},
      YEAR = {1990},
    NUMBER = {3},
     PAGES = {277--293},
     CODEN = {MSMHB2},
   MRCLASS = {13H10 (05C05 05E25 52B20)},
  MRNUMBER = {1031197},
MRREVIEWER = {Aron Simis},
}

@article {mohammad,
 AUTHOR = {Mahmoudi, Mohammad and Mousivand, Amir and Crupi, Marilena and
              Rinaldo, Giancarlo and Terai, Naoki and Yassemi, Siamak},
     TITLE = {Vertex decomposability and regularity of very well-covered
              graphs},
   JOURNAL = {J. Pure Appl. Algebra},
  FJOURNAL = {Journal of Pure and Applied Algebra},
    VOLUME = {215},
      YEAR = {2011},
    NUMBER = {10},
     PAGES = {2473--2480},
     CODEN = {JPAAA2},
   MRCLASS = {13F20 (05C75 13H10)},
  MRNUMBER = {2793950},
MRREVIEWER = {Amir Mafi},
}

@article {BHZ18,
    AUTHOR = {Bigdeli, Mina and Herzog, J\"urgen and Zaare-Nahandi, Rashid},
     TITLE = {On the index of powers of edge ideals},
   JOURNAL = {Comm. Algebra},
  FJOURNAL = {Communications in Algebra},
    VOLUME = {46},
      YEAR = {2018},
    NUMBER = {3},
     PAGES = {1080--1095},
      ISSN = {0092-7872,1532-4125},
   MRCLASS = {13D02 (05E40 13C13)},
  MRNUMBER = {3780221},
MRREVIEWER = {Siamak\ Yassemi},
       DOI = {10.1080/00927872.2017.1339058},
       URL = {https://doi.org/10.1080/00927872.2017.1339058},
}

@article {ProvLouis,
    AUTHOR = {Provan, J. Scott and Billera, Louis J.},
     TITLE = {Decompositions of simplicial complexes related to diameters of
              convex polyhedra},
   JOURNAL = {Math. Oper. Res.},
  FJOURNAL = {Mathematics of Operations Research},
    VOLUME = {5},
      YEAR = {1980},
    NUMBER = {4},
     PAGES = {576--594},
      ISSN = {0364-765X},
   MRCLASS = {52A25 (90C05)},
  MRNUMBER = {593648},
MRREVIEWER = {J. Parida},
       DOI = {10.1287/moor.5.4.576},
       URL = {https://doi.org/10.1287/moor.5.4.576},
}

@article{KTY,
    AUTHOR = {Kimura, Kyouko and Terai, Naoki and Yassemi, Siamak},
     TITLE = {The projective dimension of the edge ideal of a very
              well-covered graph},
   JOURNAL = {Nagoya Math. J.},
  FJOURNAL = {Nagoya Mathematical Journal},
    VOLUME = {230},
      YEAR = {2018},
     PAGES = {160--179},
      ISSN = {0027-7630,2152-6842},
   MRCLASS = {13D02 (05C25 13F55)},
  MRNUMBER = {3798623},
MRREVIEWER = {Margherita\ Barile},
       DOI = {10.1017/nmj.2017.7},
       URL = {https://doi.org/10.1017/nmj.2017.7},
}

@book {J08,
    AUTHOR = {Jonsson, Jakob},
     TITLE = {Simplicial complexes of graphs},
    SERIES = {Lecture Notes in Mathematics},
    VOLUME = {1928},
 PUBLISHER = {Springer-Verlag, Berlin},
      YEAR = {2008},
     PAGES = {xiv+378},
      ISBN = {978-3-540-75858-7},
   MRCLASS = {05C10 (05C20 05C40 05E30 06A11 55U10)},
  MRNUMBER = {2368284},
       DOI = {10.1007/978-3-540-75859-4},
       URL = {https://doi.org/10.1007/978-3-540-75859-4},
}

@article{SF24inc,
  author  = {Seyed Fakhari, S. A.},
  title   = {An increasing normalized depth function},
  journal = {J. Commut. Algebra},
  volume  = {16}, number = {4}, year = {2024}, pages = {497--499},
  doi     = {10.1216/jca.2024.16.497}}

@book {ms2005,
    AUTHOR = {Miller, Ezra and Sturmfels, Bernd},
     TITLE = {Combinatorial commutative algebra},
    SERIES = {Graduate Texts in Mathematics},
    VOLUME = {227},
 PUBLISHER = {Springer-Verlag, New York},
      YEAR = {2005},
     PAGES = {xiv+417},
   MRCLASS = {13-01 (05-01 05E99 13D02 14M15 14M25)},
  MRNUMBER = {2110098},
MRREVIEWER = {Joseph Gubeladze},
}

@book {west,
    AUTHOR = {West, Douglas B.},
     TITLE = {Introduction to graph theory},
 PUBLISHER = {Prentice Hall, Inc., Upper Saddle River, NJ},
      YEAR = {1996},
     PAGES = {xvi+512},
   MRCLASS = {05-01},
  MRNUMBER = {1367739},
}

@article {crupi,
    AUTHOR = {Crupi, Marilena and Rinaldo, Giancarlo and Terai, Naoki},
     TITLE = {Cohen-{M}acaulay edge ideal whose height is half of the number
              of vertices},
   JOURNAL = {Nagoya Math. J.},
  FJOURNAL = {Nagoya Mathematical Journal},
    VOLUME = {201},
      YEAR = {2011},
     PAGES = {117--131},
   MRCLASS = {05C25 (05C75 13F20 13H10 55U10)},
   MRNUMBER = {2772171},
   MRREVIEWER = {Amir Mafi},
}

@article {duval,
    AUTHOR = {Duval, Art M.},
     TITLE = {Algebraic shifting and sequentially {C}ohen-{M}acaulay
              simplicial complexes},
   JOURNAL = {Electron. J. Combin.},
  FJOURNAL = {Electronic Journal of Combinatorics},
    VOLUME = {3},
      YEAR = {1996},
    NUMBER = {1},
     PAGES = {Research Paper 21, approx.\ 14 pp.\ (electronic)},
   MRCLASS = {06A08 (52B20)},
  MRNUMBER = {1399398},
MRREVIEWER = {Hideaki Morita},
}

@article {CF23,
    AUTHOR = {Crupi, Marilena and Ficarra, Antonino},
     TITLE = {Very well-covered graphs by {B}etti splittings},
   JOURNAL = {J. Algebra},
  FJOURNAL = {Journal of Algebra},
    VOLUME = {629},
      YEAR = {2023},
     PAGES = {76--108},
      ISSN = {0021-8693,1090-266X},
   MRCLASS = {13D02 (05C75 13A70 13F55 13H10 13P10)},
  MRNUMBER = {4575657},
MRREVIEWER = {Dariush\ Kiani},
       DOI = {10.1016/j.jalgebra.2023.03.033},
       URL = {https://doi.org/10.1016/j.jalgebra.2023.03.033},
}

@article {CF24,
    AUTHOR = {Crupi, Marilena and Ficarra, Antonino},
     TITLE = {Very well-covered graphs via the {R}ees algebra},
   JOURNAL = {Mediterr. J. Math.},
    VOLUME = {21},
      YEAR = {2024},
    NUMBER = {4},
     PAGES = {Paper No. 135},
}

@article {CDRS25,
    AUTHOR = {Chau, Trung and Das, Kanoy Kumar and Roy, Amit and Saha, Kamalesh},
     TITLE = {Admissible matchings and the {C}astelnuovo-{M}umford regularity of square-free powers},
   JOURNAL = {preprint},
      YEAR = {2025},
      NOTE = {arXiv:2504.11941},
}

@article {Ziv09,
    AUTHOR = {{\v{Z}}ivaljevi{\'c}, Rade T.},
     TITLE = {Combinatorial groupoids, cubical complexes, and the {L}ov\'asz conjecture},
   JOURNAL = {Discrete Comput. Geom.},
    VOLUME = {41},
      YEAR = {2009},
     PAGES = {135--161},
}
\end{document}